\documentclass[11pt,reqno]{preprint}
\usepackage[full]{textcomp}
\usepackage[osf]{newtxtext}
\usepackage{comment}

\usepackage{amssymb}
\usepackage{mathtools}
\usepackage{hyperref}
\usepackage{breakurl}
\usepackage{mhenvs}
\usepackage{mhequ}
\usepackage{mhsymb}
\usepackage{booktabs}
\usepackage{tikz}
\usepackage{mathrsfs}
\usepackage{longtable}
\usepackage{halloweenmath}
\usepackage{scalerel}
\usepackage{cprotect}
\usepackage{multirow}
\usepackage{colortbl}
\usepackage{makecell}
\usepackage{resizegather}

\usepackage{comment}
\usepackage{wasysym}
\usepackage{centernot}
\usepackage{enumitem}
\usepackage{stackrel}

\usetikzlibrary{shapes.misc}
\usetikzlibrary{shapes.symbols}
\usetikzlibrary{shapes.geometric}
\usetikzlibrary{decorations}
\usetikzlibrary{decorations.markings}
\usetikzlibrary{patterns} 
\usetikzlibrary{snakes} 

\usepackage{subfig}
\usetikzlibrary{plotmarks}
\usetikzlibrary{decorations.pathreplacing}
\usetikzlibrary{decorations.pathmorphing}
\usetikzlibrary{calc}
\usetikzlibrary{external}
\usetikzlibrary{arrows.meta}

\let\oldskull\skull
\def\skull{\mathord{\oldskull}}

\DeclareMathAlphabet{\mathbbm}{U}{bbm}{m}{n}

\DeclareFontFamily{U}{BOONDOX-calo}{\skewchar\font=45 }
\DeclareFontShape{U}{BOONDOX-calo}{m}{n}{
  <-> s*[1.05] BOONDOX-r-calo}{}
\DeclareFontShape{U}{BOONDOX-calo}{b}{n}{
  <-> s*[1.05] BOONDOX-b-calo}{}
\DeclareMathAlphabet{\mcb}{U}{BOONDOX-calo}{m}{n}
\SetMathAlphabet{\mcb}{bold}{U}{BOONDOX-calo}{b}{n}

\setlist{noitemsep,topsep=4pt}

\makeatletter
\def\DeclareSymbol#1#2#3{%
	\expandafter\gdef\csname MH@symb@#1\endcsname{\tikzsetnextfilename{symbol#1}%
	\tikz[baseline=#2,scale=0.15,draw=symbols,line join=round,line cap=round]{#3}}%
	\expandafter\gdef\csname MH@symb@#1s\endcsname{\scalebox{0.75}{\tikzsetnextfilename{symbol#1}%
	\tikz[baseline=#2,scale=0.15,draw=symbols,line join=round,line cap=round]{#3}}}%
	\expandafter\gdef\csname MH@symb@#1ss\endcsname{\scalebox{0.65}{\tikzsetnextfilename{symbol#1}%
	\tikz[baseline=#2,scale=0.15,draw=symbols,line join=round,line cap=round]{#3}}}%
	}
\def\<#1>{\ifthenelse{\boolean{mmode}}{\mathchoice{\csname MH@symb@#1\endcsname}{\csname MH@symb@#1\endcsname}{\csname MH@symb@#1s\endcsname}{\csname MH@symb@#1ss\endcsname}}{\csname MH@symb@#1\endcsname}}
\makeatother

\makeatletter 
\newcommand*{\bigcdot}{}
\DeclareRobustCommand*{\bigcdot}{%
  \mathbin{\mathpalette\bigcdot@{}}%
}
\newcommand*{\bigcdot@scalefactor}{.5}
\newcommand*{\bigcdot@widthfactor}{1.15}
\newcommand*{\bigcdot@}[2]{%
  \sbox0{$#1\vcenter{}$}
  \sbox2{$#1\cdot\m@th$}%
  \hbox to \bigcdot@widthfactor\wd2{%
    \hfil
    \raise\ht0\hbox{%
      \scalebox{\bigcdot@scalefactor}{%
        \lower\ht0\hbox{$#1\bullet\m@th$}%
      }%
    }%
    \hfil
  }%
}
\makeatother

\def\act{\bigcdot}

\newcommand{\cut}{\mathfrak{C}}

\newcommand{\mrd}{\mathop{}\!\mathrm{d}}

\def\SNorm{\Sigma}
\def\SState{\mathcal{S}^\ym}

\newcommand{\mcH}{\mathcal{H}}

\newcommand{\mcC}{\mathcal{C}}

\newcommand{\mcL}{\mathcal{L}}

\newcommand{\mcI}{\mathcal{I}}
\newcommand{\mcF}{\mathcal{F}}
\newcommand{\mcY}{\mathcal{Y}}

\newcommand{\mcX}{\mathcal{X}}

\newcommand{\T}{\mathbf{T}}

\def\err{\mathrm{err}}

\newcommand{\U}{\mathrm{U}}

\def\${|\!|\!|}

\def\id{\mathrm{id}}

\def\var#1{#1\textnormal{-var}}

\def\Hol#1{#1\textnormal{-H{\"o}l}}
\def\gr#1{#1\textnormal{-gr}}

\DeclareMathOperator{\hol}{\mathrm{hol}}

\def\scal#1{{\langle#1\rangle}}

\def\CR{\mathcal{R}}
\def\CF{\mathcal{F}}

\def\bone{\mathbf{1}}

\def\bgamma{\boldsymbol{\gamma}}
\def\bzeta{\boldsymbol{\zeta}}

\def\Poly{\mathop{\mathrm{Poly}}}

\def\sol{{\mathop{\mathrm{sol}}}}

\newcommand{\mfu}{\mathfrak{u}}

\newcommand{\mfR}{\mathfrak{R}}

\newcommand{\mfp}{\mathfrak{p}}

\newcommand{\mfG}{\mathfrak{G}}

\newcommand{\mfg}{\mathfrak{g}}

\def\cC{\mathscr{C}}

\def\cD{\mathscr{D}}

\def\YM{\textnormal{\tiny YM}}
\def\Higgs{\textnormal{\tiny {Higgs}}}

\def\H{\textnormal{\tiny {H}}}

\newcommand{\Ad}{\mathrm{Ad}}

\DeclareMathOperator{\Trace}{Tr}

\newcommand{\SYMH}{\mathrm{SYMH}}

\def\combplus[#1,#2,#3,#4]{\binom{#1\ {\scriptstyle #4} }{#2\ #3}}

\def\singlescalegenvert[#1,#2]{\hat{H}^{#2}_{#1}}
\def\multiscalegenvert[#1,#2]{H^{#2}_{#1}}

\def\moll{\chi}

\def\nr[#1]{\tilde{N}[#1]} 
\def\inn[#1]{\mathring{N}[#1]}
\def\nrinn[#1]{\hat{N}_{#1}} 
\def\nrmod[#1,#2]{\tilde{N}_{#1}(#2)}
\def\nrinnmod[#1,#2]{\hat{N}_{#1}(#2)}

\def\ident[#1]{\underline{#1}}

\def\mylink#1#2{\mathrel{\vbox{\offinterlineskip\ialign{%
    \hfil##\hfil\cr
    $\scriptscriptstyle#1$\cr
    \noalign{\kern0.1ex}
    $#2$\cr
}}}}
\def\mysublink[#1]#2#3{\mathrel{\vbox{\offinterlineskip\ialign{%
    \hfil##\hfil\cr
    $\scriptscriptstyle#2$\cr
    \noalign{\kern0.1ex}
    $#3$\cr
    \noalign{\kern-0.2ex}
    \smash{\raisebox{-\height}{\hbox{$\scriptscriptstyle #1$}}}\cr
    \noalign{\kern0.2ex}
}}}}

\def\fon[#1]{\cC_{#1}}

\def\mincompproj[#1]{\mfp_{#1}}

\def\Proj_#1{\mathop{\mathrm{Proj}_{#1}}}

\def\negrenorm[#1]{\mfR_{#1}}
\def\topnegrenorm[#1]{\overline{\mfR}_{#1}}

\def\quotedge[#1]{E^{q}_{#1}}

\def\posrenorm[#1]{\mcC_{#1}}
\def\topposrenorm[#1]{\overline{\mcC_{#1}}}
\def\cutsmod[#1]{\mathbb{C}_{+,#1}}

\def\fullcutsmod[#1]{\cut_{#1}}

\newcommand{\ym}{{\textnormal{\scriptsize \textsc{ym}}}}
\newcommand{\ymh}{{\textnormal{\scriptsize \textsc{ymh}}}}

\colorlet{symbols}{blue!30!black!50}
\colorlet{testcolor}{green!60!black}
\colorlet{darkblue}{blue!60!black}
\colorlet{darkgreen}{green!60!black}
\definecolor{darkergreen}{rgb}{0.0, 0.5, 0.0}

\definecolor{purple}{rgb}{0.55,0.05,0.8}

\colorlet{redkernel}{red!80}

\def\symbol#1{{\mathbf{#1}}}

\def\1{\mathbf{\symbol{1}}}

\makeatletter
\pgfdeclareshape{crosscircle}
{
  \inheritsavedanchors[from=circle] 
  \inheritanchorborder[from=circle]
  \inheritanchor[from=circle]{north}
  \inheritanchor[from=circle]{north west}
  \inheritanchor[from=circle]{north east}
  \inheritanchor[from=circle]{center}
  \inheritanchor[from=circle]{west}
  \inheritanchor[from=circle]{east}
  \inheritanchor[from=circle]{mid}
  \inheritanchor[from=circle]{mid west}
  \inheritanchor[from=circle]{mid east}
  \inheritanchor[from=circle]{base}
  \inheritanchor[from=circle]{base west}
  \inheritanchor[from=circle]{base east}
  \inheritanchor[from=circle]{south}
  \inheritanchor[from=circle]{south west}
  \inheritanchor[from=circle]{south east}
  \inheritbackgroundpath[from=circle]
  \foregroundpath{
    \centerpoint%
    \pgf@xc=\pgf@x%
    \pgf@yc=\pgf@y%
    \pgfutil@tempdima=\radius%
    \pgfmathsetlength{\pgf@xb}{\pgfkeysvalueof{/pgf/outer xsep}}%
    \pgfmathsetlength{\pgf@yb}{\pgfkeysvalueof{/pgf/outer ysep}}%
    \ifdim\pgf@xb<\pgf@yb%
      \advance\pgfutil@tempdima by-\pgf@yb%
    \else%
      \advance\pgfutil@tempdima by-\pgf@xb%
    \fi%
    \pgfpathmoveto{\pgfpointadd{\pgfqpoint{\pgf@xc}{\pgf@yc}}{\pgfqpoint{-0.707107\pgfutil@tempdima}{-0.707107\pgfutil@tempdima}}}
    \pgfpathlineto{\pgfpointadd{\pgfqpoint{\pgf@xc}{\pgf@yc}}{\pgfqpoint{0.707107\pgfutil@tempdima}{0.707107\pgfutil@tempdima}}}
    \pgfpathmoveto{\pgfpointadd{\pgfqpoint{\pgf@xc}{\pgf@yc}}{\pgfqpoint{-0.707107\pgfutil@tempdima}{0.707107\pgfutil@tempdima}}}
    \pgfpathlineto{\pgfpointadd{\pgfqpoint{\pgf@xc}{\pgf@yc}}{\pgfqpoint{0.707107\pgfutil@tempdima}{-0.707107\pgfutil@tempdima}}}
  }
}
\makeatother

\colorlet{greennode}{green!50!black}
\colorlet{rednode}{red!50!black}
\colorlet{lbluenode}{blue!25}
\colorlet{dbluenode}{blue}
\colorlet{orangenode}{orange}

\definecolor{connection}{rgb}{0.7,0.1,0.1}

\tikzset{
dot/.style={circle,fill=black,inner sep=0pt, minimum size=1mm},
root/.style={circle,fill=black!50,inner sep=0pt, minimum size=3mm},
        var/.style={circle,fill=black!10,draw=black,inner sep=0pt, minimum size=1.6mm},
        delta/.style={densely dotted},
        var1/.style={rectangle,fill=black!10,draw=black,inner sep=0pt, minimum size=1.6mm},
        var2/.style={diamond,fill=black!10,draw=black,inner sep=0pt, minimum size=2mm},
        kernel/.style={semithick,shorten >=2pt,shorten <=2pt},
       kernel1/.style={postaction={decorate,decoration={markings,mark=at position 0.45 with {\draw[-] (0,-0.08) -- (0,0.08);}}}},
        kernels/.style={snake=snake,segment amplitude=1pt,segment length=4pt},
        rho/.style={densely dashed,semithick,shorten >=2pt,shorten <=2pt},
           testfcn/.style={dotted,semithick,shorten >=2pt,shorten <=2pt},
           tau/.style={circle,inner sep=1pt,draw=black,fill=white,text=black,thin},
        renorm/.style={shape=circle,fill=white,inner sep=1pt},
        labl/.style={shape=rectangle,fill=white,inner sep=1pt},
        xi/.style={very thin,circle,fill=lbluenode,draw=symbols,inner sep=0pt,minimum size=1.2mm},
        xi1/.style={very thin,rectangle,fill=lbluenode,draw=symbols,inner sep=0pt,minimum size=1.2mm},
        xi2/.style={very thin,diamond,fill=lbluenode,draw=symbols,inner sep=0pt,minimum size=1.6mm},
        xigreen/.style={very thin,circle,fill=greennode,draw=symbols,inner sep=0pt,minimum size=1.2mm},
        xigreen1/.style={very thin,rectangle,fill=greennode,draw=symbols,inner sep=0pt,minimum size=1.2mm},
        xired/.style={very thin,circle,fill=rednode,draw=symbols,inner sep=0pt,minimum size=1.2mm},
        xilblue/.style={very thin,circle,fill=lbluenode,draw=symbols,inner sep=0pt,minimum size=1.2mm},
        xidblue/.style={very thin,circle,fill=dbluenode,draw=symbols,inner sep=0pt,minimum size=1.2mm},
        xiorange/.style={very thin,circle,fill=orangenode,draw=symbols,inner sep=0pt,minimum size=1.2mm},
        xix/.style={crosscircle,fill=lbluenode,draw=symbols,inner sep=0pt,minimum size=1.2mm},
xix-green-red/.style={circle, fill=greennode!70!white,draw=rednode,inner sep=0pt,minimum size=1.6mm,append after command={node [inner sep=0pt,minimum size=0.8mm,thick, draw = rednode, cross out]{}}},
xix-green-red1/.style={rectangle, fill=greennode!70!white,draw=rednode,inner sep=0pt,minimum size=1.5mm,append after command={node [inner sep=0pt,minimum size=1mm,thick, draw = rednode, cross out]{}}},
	xib/.style={very thin,circle,fill=lbluenode,draw=symbols,inner sep=0pt,minimum size=1.6mm},
	xib1/.style={very thin,rectangle,fill=lbluenode,draw=symbols,inner sep=0pt,minimum size=1.6mm},
	xie/.style={very thin,circle,fill=greennode,draw=symbols,inner sep=0pt,minimum size=1.6mm},
	xid/.style={very thin,circle,fill=lbluenode,draw=symbols,inner sep=0pt,minimum size=1.6mm},
	xibx/.style={crosscircle,fill=lbluenode,draw=symbols,inner sep=0pt,minimum size=1.6mm},
	kernels2/.style={ultra thick,draw=symbols,segment length=12pt},
	not/.style={thin,regular polygon, regular polygon sides=3,draw=connection,fill=connection,inner sep=0pt,minimum size=1.2mm},
	notlblue/.style={thin,regular polygon, regular polygon sides=3,draw=lbluenode,fill=lbluenode,inner sep=0pt,minimum size=1.2mm},
	notorange/.style={thin,regular polygon, regular polygon sides=3,draw=orangenode,fill=orangenode,inner sep=0pt,minimum size=1.2mm},
	notgreen/.style={thin,regular polygon, regular polygon sides=3,draw=greennode,fill=greennode,inner sep=0pt,minimum size=1.2mm},
	>=stealth,
  }

\colorlet{darkblue}{blue!90!black}
\colorlet{darkred}{red!90!black}
\colorlet{darkgreen}{green!70!black}

\def\A{\mathsf{A}}

\def\s{\mathfrak{s}}

\newcommand{\e}{\varepsilon}

\def\K{\mathfrak{K}}

\def\${|\!|\!|}

\def\?{{\color{red}?}}

\def\id{\mathrm{id}}

\def\id{\mathrm{id}}

\def\bPsi{\boldsymbol{\Psi}}

\newcommand{\higgsvec}{\mathbf{V}}

\def\dash{\leavevmode\unskip\kern0.18em--\penalty\exhyphenpenalty\kern0.18em}
\def\slash{\leavevmode\unskip\kern0.15em/\penalty\exhyphenpenalty\kern0.15em}

\newcommand{\floor}[1]{\lfloor #1 \rfloor}

\usepackage{stmaryrd}
\def\fancynorm#1{{\talloblong #1 \talloblong}}
\def\quadnorm#1{{\llbracket #1 \rrbracket}}
\def\heatgr#1{{|\!|\!| #1 |\!|\!|}}

\def\state{\mathcal{S}}

\newtheorem{example}[lemma]{Example}

\let\basepoint\logof
\def\logof{\mathord{{\basepoint}}} 

\title{Uniqueness of gauge covariant renormalisation of stochastic 3D Yang--Mills--Higgs}
\author{Ilya Chevyrev$^{1}$, Hao Shen$^2$}

\institute{SISSA, \email{ichevyrev@gmail.com} \and University of Wisconsin-Madison, \email{pkushenhao@gmail.com}}

\begin{document}
\maketitle
\begin{abstract}
Local solutions to the 3D stochastic quantisation equations of Yang--Mills--Higgs were constructed in \cite{CCHS_3D}, and it was shown that, in the limit of smooth mollifications, there exists a mass renormalisation of the Yang--Mills field such that the solution is gauge covariant.
In this paper we prove uniqueness of the mass renormalisation that leads to gauge covariant solutions.
This strengthens the main result of \cite{CCHS_3D},
and is potentially important for the identification of the limit of other approximations, such as lattice dynamics.
Our proof relies on systematic short-time expansions of singular stochastic PDEs and of regularised Wilson loops. We also strengthen the recently introduced state spaces of \cite{Sourav_flow,Sourav_state,CCHS_3D} to allow finer control on line integrals appearing in expansions of Wilson loops.
\end{abstract}
\setcounter{tocdepth}{2}

\tableofcontents

\section{Introduction}

In this paper, we study the 3D stochastic quantisation equation (i.e. Langevin dynamic) of Yang--Mills--Higgs (SYMH) with DeTurck term.
To simplify the discussion, we suppose there is no Higgs field, in which case the equation reads
\begin{equ}[eq:SYM]
\partial_t A_i = 
\Delta A_i  + [A_j,2\partial_j A_i - \partial_i A_j + [A_j,A_i]] 
+ (C^{\eps}_{\A} A)_i +  \xi^\eps_i
\end{equ}
and is posed for $A = (A_1,A_2,A_3)\colon [0,\infty)\times\T^3 \to\mfg^3$, where $\mfg$ is the Lie algebra of a compact Lie group and $\T^3$ is the 3D torus,
$\xi^\e$ is a mollification at scale $\e>0$ of a $\mfg^3$-valued white noise on $\R\times\T^3$,
and $C^\e_\A \in L(\mfg^3,\mfg^3)$ is a `renormalisation operator'.
We refer to Section \ref{sec:setup} for precise definitions and for the generalisation that includes a Higgs field.

One of the main results of \cite{CCHS_3D} is that there exist operators $C^\e_\A \in L(\mfg^3,\mfg^3)$ such that the solution to \eqref{eq:SYM}
converges locally in time as $\eps\downarrow0$
and the limiting dynamic is gauge covariant
in the following sense: if $A,\bar A$ are two limiting dynamics started from gauge equivalent initial conditions $A(0) \sim \bar A(0)$,
then $A(t)$ and $\bar A(t)$ are gauge equivalent \emph{in law} for all $t\geq 0$ (modulo possible finite-time blow-up).
See Definition \ref{def:state} and Theorem \ref{theo:meta} for a summary of these results.
This allows, in particular, to construct a canonical Markov process on gauge orbits associated to the SYMH, which conjecturally has a unique invariant measure.
(Proving that this invariant measure exists would yield a construction of the 3D YMH measure on $\T^3$, which is currently an open problem.)

In this paper, we address the question of \emph{uniqueness} of the operators $\{C^\e_\A\}_{\e>0}$.
Our main result (or rather a corollary of it) is that, if $C^\e_\A, \bar C^\e_\A\in L(\mfg^3,\mfg^3)$ are renormalisation 
operators which both render the limiting dynamic of \eqref{eq:SYM} gauge covariant in the sense described
above,\footnote{in this case we say that they are `gauge covariant renormalisations' or `gauge covariant operators'}
then
\begin{equ}
\lim_{\eps\downarrow0} |C^\e_\A - \bar C^\e_\A| = 0\;.
\end{equ}
In fact, suppose $C^\e_\A$ are the gauge covariant operators from \cite{CCHS_3D} and that $\lim_{\eps\downarrow0} C^\e_\A - \bar C^\e_\A$ is non-zero.
Let $A^a$ be the limiting dynamic of \eqref{eq:SYM} with $\bar C^\e_\A$ and with initial condition $A(0)=a$.
Then, for all $t>0$ sufficiently small, one can find gauge equivalent $a\sim b$ and $s>0$ such that
\begin{equ}[eq:sep]
|\E W_\ell [\CF_s (A_t^a)] - \E W_\ell [\CF_s (A_t^b)]| \gtrsim t^{\frac{10}{9}}
\end{equ}
(the exponent $\frac{10}{9}$ can be replaced by any number larger than $1$).
Here, $W_\ell [\CF_s (\cdot)]$ is a \emph{regularised Wilson loop}, i.e. $W_\ell$ is a classical Wilson loop and $\CF_s(A)$ is a regularisation of $A$ given by the time-$s$ Yang--Mills (YM) heat flow started from $A$.
Since $W_\ell [\CF_s (\cdot)]$ is gauge invariant, the estimate \eqref{eq:sep} quantifies `non-gauge covariance' of SYMH renormalised with $\bar C^\e_\A$.
We refer to Theorem \ref{theo:main} for a detailed statement, which incorporates a Higgs field.
See also Section \ref{sec:idea} for a description of the main steps in the proof and the challenges that arise.

\textbf{Motivation and related results.}
Our motivation for this study is threefold.
First, our results contribute to the understanding of the symmetries of SYMH, in which we believe there is intrinsic interest.
They in particular help justify the fact that the Markov process constructed in \cite{CCHS_3D} is canonical.
For recent results on symmetries of other geometric SPDEs, see e.g. \cite{BGHZ22,Bruned_Dotsenko_24,Bellingeri_Bruned_24}. 

Second, as described in \eqref{eq:sep}, our main result separates expectations of regularised gauge invariant observables (Wilson loops).
The idea to regularise fields via the YM heat flow appeared in earlier works such as \cite{NN06,Luscher10,CG13}.
Recently, \cite{Sourav_flow,Sourav_state} and \cite{CCHS_3D},
proposed a state space for the 3D YM measure based on the YM heat flow
(see also \cite{Chevyrev22_norm_inf,Chevyrev_Mirsajjadi_24}).
Our analysis actually reveals an interplay between the small time $t>0$ of the SPDE and the `regularisation time' $s>0$ (we take $s$ as a small, but not too small, power of $t$)
and demonstrates a quantitative property that can be extracted from these observables as we send the regularisation scale $s\downarrow0$.

Finally, a similar result was shown in \cite{chevyrev2023invariant} in the simpler setting of $\T^2$ (see Theorem 8.1 therein and Remark \ref{rem:2D_compare} below),
and this result was used in the proof of universality of the continuum 2D YM Langevin dynamic studied first in \cite{CCHS_2D}.
Roughly speaking, \cite{chevyrev2023invariant} first showed tightness of many dynamical lattice models and that any limit point is a solution of \eqref{eq:SYM} with \textit{some} renormalisation $\mathring C_\A$.
To prove that there is only one limit point, the argument in
\cite{chevyrev2023invariant} is to remark that every limit is gauge covariant, from which uniqueness of the limit follows from uniqueness of the gauge covariant $\mathring C_\A$.
Deriving the scaling limit of the 3D lattice YMH Langevin dynamic is currently an open problem,
but we expect that the results of this paper may similarly help to establish uniqueness of limit points.

\begin{remark}
Compared to \cite[Sec.~8]{chevyrev2023invariant}, the proof of the present result is more involved because the solution of \eqref{eq:SYM} is much more singular in 3D than in 2D.
We thus require a more systematic small-time expansion of the SPDE and an analysis of the YM heat flow $\CF_s$ (which is not needed in 2D). We moreover require a more delicate choice of initial conditions $a,b$ in \eqref{eq:sep} than in \cite{chevyrev2023invariant}.
\end{remark}
In the rest of the introduction, we describe in detail our main result.

\subsection{Background and setup}
\label{sec:setup}

We say that $\moll\in\CC^\infty(\R\times\R^3)$ is a \emph{mollifier} if it has support in $\{z\,:\,|z|\leq 1/4\}$ where $|\cdot|$ is the parabolic distance on $\R\times\R^3$
and $\int\moll=1$.
For $\eps\in(0,1]$ and a mollifier $\moll$, write
\begin{equ}
\moll^{\eps}(t,x) =\eps^{-5} \moll(\eps^{-2}t,\eps^{-1}x)
\end{equ}
and $\xi^\eps = \moll^\eps * \xi$ for a distribution $\xi \in\CD'(\R\times \T^3)$.
We say that $\moll$ is \emph{non-anticipative }if it has support in $\{(t,x)\in\R\times\R^3\,:\, t>0\}$.

Let $G$ be a connected compact Lie group and $\mfg$ its Lie algebra.
We assume without loss of generality that $G\subset \U(N)$ and $\mfg\subset\mfu(N)$ for some $N\ge 1$.\label{page:G}\label{page:mfg}
We denote by $1\in G$ the identity element.
Let $(\higgsvec,\scal{\cdot,\cdot}_\higgsvec)$\label{page:higgsvec} be a real Hilbert space with an orthogonal left group action of $G$ that we write by $\higgsvec \ni v \mapsto g v \in \higgsvec$.
Note that we allow $\higgsvec=\{0\}$, which is called the \emph{pure} YM model.

We write the corresponding representation of the Lie algebra similarly as $v \mapsto Av$ for $A\in\mfg$.
We equip $\mfg$ with an $\Ad$-invariant inner product $\scal{\cdot,\cdot}_{\mfg}$.

Throughout the article, unless otherwise stated, we denote\label{page:E}
\begin{equ}
E=\mfg^3\oplus\higgsvec\;.
\end{equ}
For an $E$-valued distribution $X=(A,\Phi)$, we write $X^\ym=A$\label{page:ym}
for the $\mfg^3$-component.

As an example, one can take $G= \U(N)$, $\mfg = \mfu(N)$, and $\higgsvec = \C^N$, with the natural representation and with inner products $\scal{x,y}_{\higgsvec} = \Re\sum_{i=1}^N x_i\bar{y_i}$
and $\scal{A,B} = \Trace(AB^*)$,
where $\Trace$ is the trace.

For a vector space $F$, we let $L(F) = L(F,F)$\label{page:L} be the set of linear maps from $F$ to itself.
If $F$ carries a group action of $G$, we write $L_G(F)\subset L(F)$ for the space of linear operators that commute with the action of $G$.
We equip $\mfg$ and $\mfg^3$ with the adjoint actions, given for $g\in G$ by
\begin{equ}
\mfg\ni v\mapsto \Ad_g v \in\mfg
\quad 
\text{ and }
\quad \mfg^3\ni (v_1,v_2,v_3) \mapsto (\Ad_g v_1, \Ad_g v_2,\Ad_g v_3) \in \mfg^3\;.
\end{equ}
Recall the BPHZ constants from \cite[Remark~1.8, Proposition~5.7]{CCHS_3D}\label{page:BPHZ}
\begin{equ}
C_{\YM}^{\eps}\in L_G(\mfg)\;,
\qquad
C^\eps_{\Higgs}\in L_G(\higgsvec)
\end{equ}
which, in general, depend on $\moll$.
Unless otherwise stated, we extend an operator $c\in L(\mfg)$ to $c \in L(\mfg^3)$  
block diagonally. 

%
%

Consider furthermore $\mathring{C}_{\A}\in L(\mfg^3)$
and $\mathring C_{\Phi} \in L(\higgsvec)$\label{page:C_bare} and denote\footnote{We choose to use the clearer notation $\mathring C_{\Phi}$ here, which corresponds to  $-m^2$ in  \cite{CCHS_3D}.}\label{page:C_eps}
\begin{equ}[eq:C_eps]
C^{\eps}_{\A} = C_{\YM}^{\eps}+\mathring{C}_{\A} \in L(\mfg^3)
\qquad
\mbox{and}
\qquad
C^\eps_{\Phi} = C^\eps_{\Higgs} + \mathring C_{\Phi} \in L(\higgsvec)\;.
\end{equ}
For a $\mfg^3$-valued white noise $\xi_\ym = (\xi_1,\xi_2,\xi_3)$, $\higgsvec$-valued white noise $\xi_{\H}$ on $\R\times\T^3$,
consider the system of SPDEs on $[0,1] \times \T^3$
with $i \in \{1,2,3\}$ 
\begin{equs}
\partial_t A_i &= 
\Delta A_i  + [A_j,2\partial_j A_i - \partial_i A_j + [A_j,A_i]] 
\\
&\qquad\qquad -\mathbf{B}((\partial_{i} \Phi + A_{i}\Phi) \otimes \Phi)+ (C^{\eps}_{\A} A)_i +  \xi^\eps_i \;,
\\
\partial_t\Phi  &= 
\Delta \Phi 
+ 2 A_{j} \partial_{j}\Phi + A_{j}^{2}\Phi  - |\Phi|^2 \Phi + C^{\eps}_{\Phi} \Phi +  \xi_\H^{\eps} \;,\\
(A(0),& \Phi(0)) = (a,\phi) \in\CC^\infty\;,
		\label{eq:SPDE_for_A}
\end{equs}
where the summation over $j$ is  implicit,
and the map $\mathbf{B}\colon \higgsvec \otimes \higgsvec \rightarrow \mathfrak{g}$  is the unique  $\R$-linear form such that,
for all $u,v \in \higgsvec$ and $h \in \mfg$,
\begin{equ}
	\scal{\mathbf{B}(u\otimes v),h}_{\mfg} =  
	\scal{u,hv}_\higgsvec\;.
\end{equ}

As in \cite{CCHS_3D,chevyrev2023invariant}, we rewrite the above equation in the following shorthand notation
\begin{equ}\label{eq:SYM_moll}
\d_t X = \Delta X + X\d X + X^3 + C^\eps X + \moll^\eps *\xi
\end{equ}
for
\begin{equ}
X = (A,\Phi)\colon [0,T] \times \T^3\to E\;,
\end{equ}
where $T\in (0,1]$ is an existence time of the SPDE, $\xi = (\xi_1,\xi_2,\xi_3,\xi_\H)$ is a $\mfg^3\times\higgsvec$-valued white noise, and
\begin{equ}[eq:def-C-family]
C = \{C^\eps\}_{\eps\in(0,1)} = \{C^\eps_\A,C^\eps_\Phi\}_{\e\in(0,1)}\;.
\end{equ}
In the above equation \eqref{eq:SYM_moll}, to lighten notation, we do not write the dependence of $X$ on $\eps$.
In the equations that follow, we will make the dependence on $\eps$ explicit whenever it becomes important (e.g. in \eqref{eq:tilde_A_equ}-\eqref{eq:PDE_g} and in the proof of Theorem \ref{theo:main} below).

\begin{definition}[Path space and solution]\label{def:state}
Recall the state space $\state$ from \cite[Sec.~2.3]{CCHS_3D}.\footnote{We recall this space briefly in Remark~\ref{rem:state_embed}.}
Let $\skull$ denote a cemetery state, and, following~\cite[Sec.~1.5.1]{CCHS_2D}, for a metric space $F$, let $F^\sol$\label{page:sol} denote the space of paths $f\colon [0,1] \to F\sqcup\{\skull\}$ that can blow-up in finite time by leaving every bounded set and cannot be `reborn'.

Let $\SYMH(C, (a,\phi)) \in \state^\sol$\label{page:SYMH}
denote the limit of the solution to \eqref{eq:SYM_moll}
with $C$ as in \eqref{eq:def-C-family} and $(a,\phi)$ as in \eqref{eq:SPDE_for_A}; this limit exists in $\state^\sol$ due to
\cite[Thm. 1.7]{CCHS_3D},
and this limit depends only on $(\mathring C_\A,\mathring C_\Phi)$ and not on $\moll$.
\end{definition}

Now we recall the main result of \cite{CCHS_3D},
which states that there is a choice of $\mathring C_\A$ such that $\SYMH(C, (a,\phi))$
is gauge covariant.
For $\rho\in[0,\infty]$,
we write\label{page:mfG}
\begin{equ}
\mfG^\rho \eqdef \CC^\rho(\T^3,G)
\end{equ}
for the   gauge group (i.e. the group of gauge transformations)
of H\"older regularity $\rho$.
By  \cite[Thm.~1.2 (iv)]{CCHS_3D},
there exists $\rho \in (\frac12,1)$ and a continuous left group action $\mfG^{\rho}\times\state \ni (g,X)\mapsto g \act X \in\state$
such that, whenever $g$ 
and $X=(A,\Phi)$ are smooth, $g\act X = (g\act A,g\act \Phi)$ is given by 
\begin{equ}[e:gauge-transformation]
	g\act A \eqdef \Ad_g(A)  - (\mrd g)g^{-1} \;,
	\qquad
	\mbox{and}
	\qquad
	g\act \Phi  \eqdef  g\Phi\;.
\end{equ}
In the following theorem we fix this $\rho \in (\frac12,1)$.

\begin{theorem}(\cite[Theorems 1.9 and 6.1]{CCHS_3D})
\label{theo:meta}
Let  $\moll$ be a  non-anticipative mollifier.
There exists a unique $\mathring{C}_{\A} \in L_{G}(\mfg)$, independent   of mollifier $\moll$, with the following property.

Let $C = \{C^\e\}_{\e\in(0,1)}$ be as in \eqref{eq:def-C-family} and \eqref{eq:C_eps}.
Then for all $g(0) \in \mathfrak{G}^{\rho}$ and $(a,\phi) \in \state$, one has, modulo finite time blow-up,
 \begin{equ}[e:gauge-cov]
 g\act
 \SYMH(C, (a,\phi))
\eqlaw
\SYMH(C, g(0)\act(a,\phi))
\end{equ}
 where  $g$ is the solution to
\begin{equ}\label{eq:SPDE_for_g_wrt_A}
g^{-1}(\partial_t g)
= \partial_j(g^{-1}\partial_j g)+ [A_j,g^{-1}\partial_j g]
\end{equ}
with  initial condition $g(0)$, where $A$ is the $\mfg^3$-component of $ \SYMH(C, (a,\phi))$.
\end{theorem}

The equation \eqref{eq:SPDE_for_g_wrt_A} for $g$ is classically well-posed in 2D, 
whereas in 3D 
it is classically ill-posed but can be given meaning using regularity structures via a limiting procedure, see \cite[Lem.~C.1]{CCHS_3D} (however, $g$ might blow up before $A$ does, and we refer to \cite[Thm.~6.1]{CCHS_3D} for the more precise statement of the above theorem).  
Throughout the paper, as in \cite[Thm.~6.1]{CCHS_3D}, we write 
\begin{equ}[eq:check_C_def]
\check C \in L_G(\mfg)
\end{equ}
to be the unique $\mathring{C}_{\A} \in L_{G}(\mfg)$ in Theorem~\ref{theo:meta}.

The main question we address in this paper is the following.
By Theorem~\ref{theo:meta}
we only know that 
$\check C$ is the unique operator 
such that \eqref{e:gauge-cov} holds with the particular 
$g$ that solves \eqref{eq:SPDE_for_g_wrt_A}.
It \emph{does not} rule out the possibility that there exists another
$\mathring{C}_{\A} \in L_G(\mfg)$ for which one still has \eqref{e:gauge-cov} for a different choice of $g$,
or, more generally, that
\begin{equ}{}
[\SYMH(C, (a,\phi))]
\eqlaw 
[\SYMH(C, g(0)\act(a,\phi))]
 \end{equ}
where $[X]$ denotes the gauge equivalence class of $X$.
In this paper we indeed rule out this possibility and thus prove an `intrinsic' notion of 
uniqueness.
Roughly speaking, in our main Theorem~\ref{theo:main}, we show that, if $\mathring C_\A \neq \check C$, then
one can use proper gauge invariant observables to 
`detect non-gauge-covariance'. 
The gauge invariant observables are Wilson loops regularised by the YM heat flow
as in \cite{CG13,Sourav_state,CCHS_3D}, 
which we recall now.

\begin{definition}[YM heat flow]\label{def:CF}
For a distribution $a \in \CD'(\T^3,\mfg^3)$
of suitable regularity,\footnote{See
Section \ref{sec:Fs} or \cite[Sec.~2.1]{CCHS_3D} for the precise definition of `suitable regularity'.
Elements of $\state$ from Definition \ref{def:state}, including $\SYMH_t(C, (a,\phi))$ for all $t \geq 0$, have this regularity.}
let $T_a>0$ denote the time of blow-up in $\CC^\infty$ of the maximal solution $A\colon (0,T_a)\to \CC^\infty(\T^3,\mfg^3)$ to the DeTurck--YM heat flow
\begin{equs}[eq:YMH_flow]
\partial_s A_i &= 
\Delta A_i  + [A_j,2\partial_j A_i - \partial_i A_j + [A_j,A_i]]\;,\\
A(0) &= a\;,
\end{equs}
and let $\CF_s(a) = A_s$\label{page:CF} denote the corresponding solution evaluated at time $s\in [0,T_a)$.
We also set $\CF_s(a)=\skull$ for $s\geq T_a$
and $\CF_s(\skull) = \skull$ for all $s\geq0$.
\end{definition}

For $x=(a,\phi)\in\CC^\infty(\T^3,E)$ and $g\in \mfG^\infty$, recall the notations $g\act x$ and $g\act a$ from \eqref{e:gauge-transformation}.
We say that $x,y\in \CC^\infty(\T^3,E)$ are gauge equivalent, and write $x\sim y$,
if $g\act x= y$ for some $g\in \mfG^\infty$.
We make similar definition for $a,b\in \CC^\infty(\T^3,\mfg^3)$.

Recall that $\sim$ is an equivalence relation on $\CC^\infty(\T^3,E)$, called gauge equivalence, which extends canonically to $\state$.\footnote{This extension is given in \cite[Def.~2.11]{CCHS_3D} and relies on the DeTurck--YMH heat flow. Moreover, $g\act X \sim X$ for all $g\in\mfG^\rho$ and $X\in\state$ as above \eqref{e:gauge-transformation}.}
Moreover, if $(a,\phi)\sim (\bar a,\bar\phi)$, then $a\sim \bar a$,
and the latter is equivalent to $\CF_s(a)\sim \CF_s(\bar a)$ for all $s\in (0,T_a\wedge T_{\bar a})$, see \cite[Prop.~2.15]{CCHS_3D}.  

\begin{definition}[Wilson loop]
\label{def:Wilson-loop}
For $\ell\in\CC^\infty([0,1],\T^3)$ and $A\in\CC^\infty(\T^3,\mfg^3)$, define the holonomy of $A$ along $\ell$ as
$\hol(A,\ell) = y_1$ where $y\colon [0,1]\to G$ is the solution to the ODE
\begin{equ}
\mrd y_t = y_t \mrd \ell_A\;,\quad y_0=1\;,
\end{equ}
and where $\ell_A \colon [0,1]\to\mfg$ is the line integral
\begin{equ}[eq:ell_A]
\ell_A(t) = \int_0^t \scal{A(\ell_s),\dot\ell_s} \mrd s\;.
\end{equ}
Whenever $\ell$ is a loop, i.e. $\ell(0)=\ell(1)$, we define the \textit{Wilson loop}\label{page:W_loop}
\begin{equ}
W_\ell(A) = \Trace \hol(A,\ell)\;.
\end{equ}
\end{definition}

Recall that, for smooth gauge equivalent $1$-forms $a\sim b$, one has $W_\ell(a)=W_\ell(b)$, i.e. $W_\ell$ is a gauge invariant function.
In particular, for $x = (a,\phi)\sim y=(b,\psi)$ in $\state$,
one has $W_\ell[\CF_s(a)] = W_\ell[\CF_s(b)]$.
(In fact, one can characterise $\sim$ via observables similar to $W_\ell [\CF_s (\cdot)]$, see \cite[Prop.~2.67]{CCHS_3D}.)

\subsection{Main result}
The following is our main result, which shows that if $\mathring C_\A$ deviates from $\check C$ by  $c\neq 0$,
then we lose gauge covariance.

\begin{theorem}\label{theo:main}
Recall $\check C \in L_G(\mfg)$ from \eqref{eq:check_C_def}.
Consider any $\mathring C_{\Phi}\in L_G(\higgsvec)$ and
a non-anticipative mollifier $\moll$.
Consider non-zero $c \in L_G(\mfg^3)$ and define
\begin{equ}
\mathring C_\A = \check C+ c\in L_G(\mfg^3)\;,\quad
C^\eps = (C^\eps_{\YM}+ \mathring C_\A, C^\eps_\Phi) \in L_G(\mfg^3)\oplus L_G(\higgsvec)\;.
\end{equ}
For $x=(a,\phi)\in \CC^\infty(\T^3,E)$, $g\in \mfG^\infty$, we write 
\begin{equ}[eq:A_bar_A_def]
\SYMH( C, x) = (A^{(1)},\Phi^{(1)})\;,\qquad \SYMH( C, g\act x) = (A^{(2)}, \Phi^{(2)})\;.
\end{equ}
Let $r>0$.
Then there exist $\sigma,t_0,\beta>0$, a loop $\ell\in\CC^\infty([0,1],\T^3)$ and $g\in \mfG^\infty$, such that, for all $t<t_0$, there exists
$x = x^{(t)}\in \CC^\infty(\T^3,E)$ such that $|x|_{\CC^3} < 1/\sigma$ and
\begin{equ}[eq:W_ell_F_s_A_diff]
|\E W_\ell [\mcF_s (A^{(1)}_t)] - \E  W_\ell [\mcF_s ( A^{(2)}_t)]| \geq \sigma t^{1+r}\;,
\end{equ}
where $s=t^\beta$ and, in case of finite-time blow-up, we define $W_\ell(\skull)=0$ by convention.
\end{theorem}

\begin{remark}\label{rem:projection}
We obtain a fixed power of $t$ (i.e. the exponent $1+r$) 
in the lower bound \eqref{eq:W_ell_F_s_A_diff} and this is important because it ensures that the difference in expectation of Wilson loops is not due to different blow-up times of $\SYMH( C, x)$ and $\SYMH( C, g\act x)$.
Indeed, for any $M>0$, $\P[\SYMH_t(C,x)=\skull] \lesssim t^M$ locally uniformly in $x\in\CC^3$.
Consequently, for $c=0$, the left-hand side of \eqref{eq:W_ell_F_s_A_diff} is bounded from above by $\lesssim t^M$ for any $M>0$.

As a corollary, it follows that there exists only one renormalised solution of SYMH in 3D that produces a Markov process on $\state/{\sim}$
from generative probability measures as in \cite[Sec.~7]{CCHS_3D}.
See \cite[Cor.~8.3]{chevyrev2023invariant} for a precise version of such a result in 2D, which, for brevity, we do not reproduce here in 3D.
\end{remark}

\begin{remark}\label{rem:2D_compare}
One should compare Theorem \ref{theo:main} with \cite[Thm.~8.1]{chevyrev2023invariant}, which
is an analogous result for $\T^2$ (and without Higgs) and with a weaker lower bound $\sigma t^2$.
However, $x$ in \cite[Thm.~8.1]{chevyrev2023invariant} is taken simply as $x=0$ and does not depend on $t$.
In contrast, our $x$ in Theorem \ref{theo:main} does depend on $t$ and will be taken of size $|x|_{\CC^3} \asymp t^{r}$ for $r>0$ small.
This choice is crucial to dominate more badly behaved remainder terms in 3D; see the beginning of Section \ref{sec:initial} for further discussion.
\end{remark}

Theorem \ref{theo:main} in fact follows from the following more general result.
For an operator $c\in L(\mfg^3)$, we write $c = (c_1,c_2,c_3)$ where $c_i\in L(\mfg^3,\mfg)$.

\begin{proposition}\label{prop:A_tilde_A}
Consider the loop $\ell\in\CC^\infty([0,1],\T^3)$, $\ell(x)=(x,0,0)$,
and let $c, \mathring C_{\A}\in L(\mfg^3)$ with $c_1\neq 0$,
and $C^\eps=(C_{\YM}^{\eps}+\mathring{C}_{\A},C^\eps_\Phi)$.

Let $r>0$.
Then there exist $t_0,\sigma>0$ depending only on $\moll,c,\mathring C_{\A},G,r$, and $g(0)\in \mfG^\infty$ depending only on $c,G$, and $\beta>0$ depending only on $r$, such that
the following holds: for all $t\in (0,t_0)$ there exists $\tilde x = \tilde x^{(t)} \in \CC^\infty(\T^3,E)$ with $|\tilde x|_{\CC^3} < 1$
such that
\begin{equ}[e:1+r]
|\E W_\ell[\mcF_s ( A_t)] - \E W_\ell [\mcF_s (\tilde A_t)]| \geq \sigma t^{1+r}
\end{equ}
where $s=t^\beta$, $X=(A,\Phi) = \SYMH(C,\tilde x)$, and $\tilde X = (\tilde A,\tilde\Phi)$ is the $\eps\downarrow0$ limit of solutions to
\begin{equ}[eq:tilde_A_equ]
\partial_t \tilde X^\eps = \Delta\tilde X^\eps + \tilde X\partial \tilde X^\eps + (\tilde X^\eps)^3 +\moll^\eps* \xi + C^\eps \tilde X^\eps + (c\mrd \tilde g^\eps  (\tilde g^\eps)^{-1},0)\;,\quad \tilde X^\eps(0)=\tilde x\;,
\end{equ}
where, writing $\tilde X^\eps = (\tilde A^\eps,\tilde\Phi^\eps)$,
$\tilde g^\eps$ solves the PDE
\begin{equ}\label{eq:PDE_g}
\partial_t \tilde g^\eps = \Delta \tilde g^\eps - (\partial_j \tilde g^\eps) (\tilde g^\eps)^{-1} (\partial_j \tilde g^\eps) + [\tilde A^\eps_j, (\partial_j \tilde g^\eps) (\tilde g^\eps)^{-1}] \tilde g^\eps
\end{equ}
with initial condition $g(0)$.
We treat the limit $(\tilde X, \tilde g) = \lim_{\eps\downarrow 0} (\tilde X^\eps,\tilde g^\eps)$
as a random variable in $(\state\times \mfG^{\rho})^\sol$ for $\rho>\frac12$
(in particular $\tilde X_t=\skull$ if either $\tilde X$ or $\tilde g$ blow up before time $t$).
\end{proposition}

\begin{remark}
Similarly to the comment after Theorem \ref{theo:meta},
the limit $(\tilde X, \tilde g) = \lim_{\eps\downarrow 0} (\tilde X^\eps,\tilde g^\eps)$ exists due to \cite[Lem.~C.1]{CCHS_3D}.
\end{remark}

\begin{remark}
For readers not familiar with  \cite{CCHS_3D}, let us give some motivation for  \eqref{eq:PDE_g}
appearing in the above proposition. 
Assuming that $(A,g)$ solves \eqref{eq:SPDE_for_g_wrt_A},
and letting $B \eqdef g\act A = \Ad_g(A)  - (\mrd g)g^{-1} $ as in \eqref{e:gauge-transformation},
then one can easily rewrite \eqref{eq:SPDE_for_g_wrt_A} in terms of the `target' (i.e. gauge transformed) process $B$
as (see also  \cite[(1.16)-(1.17)]{CCHS_3D})  
\[
(\partial_t g) g^{-1} = \partial_j ((\partial_j g) g^{-1}) + [B_j, (\partial_j g) g^{-1}]
\]
which is equivalent with   \eqref{eq:PDE_g} if $B$ is replaced by $\tilde A$.
\end{remark}

\begin{remark}
As the proof will reveal, the exact form of \eqref{eq:PDE_g} is hardly important and the same result holds for much more general $\tilde g$.
See Remark \ref{rem:h_general} for the precise conditions on $\tilde g$ that we use.
\end{remark}

We next prove Theorem \ref{theo:main} by applying Proposition~\ref{prop:A_tilde_A} with $\mathring C_{\A} = \check C + c$.
Remark, however, that in Proposition~\ref{prop:A_tilde_A},
$c$ and $\mathring C_{\A}$ are not required to be related (i.e. their difference is not required to be $\check{C}$)
and we do not require $c,\mathring C_{\A}\in L_G(\mfg^3)$.

\begin{proof}[of Theorem \ref{theo:main}]
Since $c\neq0$, we can assume without loss of generality that $c_1\neq 0$.
Let $\tilde x$ and $g(0)$ be as in Proposition \ref{prop:A_tilde_A} for some $t\in(0,t_0)$.
Consider $x$ such that $g(0)\act x = \tilde x$.

The proof below will leverage the argument leading to the proof of  \cite[Thm.~6.1]{CCHS_3D}, i.e. Theorem~\ref{theo:meta}.

Let $X^{(1),\eps} = (A^{(1),\eps},\Phi^{(1),\eps})$ be the solution to \eqref{eq:SYM_moll} with initial condition $X^{(1),\eps}(0)=x$,
so that
\[
X^{(1)} \eqdef (A^{(1)},\Phi^{(1)}) = \SYMH(C,x) = \lim_{\eps\downarrow0}X^{(1),\eps}
\;.
\]
Since $C^\eps$ commutes with the action of $G$,
recall from \cite[(1.16)-(1.17)]{CCHS_3D}
 that for each $\eps\in (0,1)$,
$X^{(1),\eps}$
 is pathwise gauge equivalent to $Y^\eps$ which solves
\begin{equ}
\d_t Y^\eps = \Delta Y ^\eps+ Y^\eps\d Y^\eps + (Y^\eps)^3 + C^\eps Y^\eps +  (C^\eps\mrd g^\eps (g^\eps)^{-1},0)+ \Ad_{g^\eps} (\moll^\eps *\xi)\;,
\end{equ}
with $Y^\eps(0) = g(0)\act X^{(1),\eps}(0) = \tilde x$.
More precisely, $g^\eps\act X^{(1),\eps} = Y^\eps$ up until the blow up of $(X^{(1),\eps},Y^\eps,g^\eps)$, 
where $g^\eps$ solves~\eqref{eq:PDE_g} with initial condition $g^\eps(0) = g(0)$ but with $\tilde A^\eps$ replaced by $(Y^\eps)^\ym$, the $\mfg^3$-component of $Y^\eps$.

Consider now the following equation for $\bar X^\eps=(\bar A^\eps,\bar \Phi^\eps)$
\begin{equs}{}
&\partial_t \bar X^\eps =\Delta \bar X^\eps + \bar X^\eps\partial \bar X^\eps + (\bar X^\eps)^3  + C^\eps \bar X^\eps + (c \mrd \bar g^\eps (\bar g^\eps)^{-1},0)+\moll^\eps* (\Ad_{\bar g^\eps} \xi)
\;,
\\
&\bar X^\eps(0)=\tilde x\;,
\end{equs}
where $\bar g^\eps$ solves~\eqref{eq:PDE_g}, but with $\tilde A^\eps$ replaced by $\bar A^\eps$,
and with initial condition $\bar g^\eps(0) = g(0)$.

Recall that $C^\eps = (C^\eps_{\A}, C^\eps_\Phi)$
where $C^\eps_\A = C^\eps_{\YM} + \mathring C_\A$ and $\mathring C_\A = \check C + c$. 
By \cite[Thm.~6.1(ii)-(iii)]{CCHS_3D},
$\check{C}\in L_{G}(\mfg)$
is the unique operator
such that, for all $\eps$-independent initial conditions $(Y(0), g(0)) = (\bar{X}(0), \bar{g}(0))$ for $(Y^\eps,g^\eps)$ and $(\bar X^\eps,\bar g^\eps)$,
and operators $\mathring C_\A\in L(\mfg^3)$,\footnote{The statement
of \cite[Thm.~6.1]{CCHS_3D} is restricted to $\mathring C_\A\in L_G(\mfg)$,
but it is simple to see that the same proof applies to any operator $\mathring C_\A\in L(\mfg^3)$,
see also \cite[Thm.~1.14]{Chevyrev22YM}.}
$(Y^\eps,g^\eps)$ and $(\bar X^\eps,\bar g^\eps)$
converge in probability to the same limit as $\eps\downarrow 0$.

Finally, for each $\eps\in (0,1)$, clearly $(\bar X^\eps,\bar g^\eps)$ 
is equal in law to $(\tilde X^\eps,\tilde g^\eps)$ which solve \eqref{eq:tilde_A_equ}-\eqref{eq:PDE_g}
with initial condition $\tilde X^\eps(0)=\tilde x$ and $\tilde g^\eps(0)=g(0)$.
Taking limits $\eps\downarrow0$ and using that gauge equivalence of $X^{(1),\eps}$ and $Y^\eps$ is preserved in this limit by \cite[Prop.~2.28]{CCHS_3D}, it follows that,
for $s=t^\beta$ and any loop $\ell$ and $M>0$,
\begin{equ}[eq:A_tilde_A_close]
\E W_\ell [\mcF_s (A^{(1)}_t)] = \E  W_\ell [\mcF_s (\tilde A_t)] + O(t^{M})
\end{equ}
where $\beta>0$ is from Proposition \ref{prop:A_tilde_A} and we write $\tilde X = (\tilde A,\tilde\Phi) = \lim_{\eps\downarrow0}\tilde X^\eps$ and where $O(t^M)$ accounts for the possibility of different blow-up times of $A^{(1)}$ and $\tilde A$ and of $\mcF (A^{(1)}_t)$ and $\mcF (\tilde A_t)$
(we use here smallness of $s=t^\beta$ and the bound $|\tilde x|_{\CC^3} < 1$ uniform in $t\in(0,t_0)$ to conclude that the probability that $\CF (A^{(1)}_t)$ and $\CF(\tilde A_t)$ blow up before time $s$ is bounded below by $1-O(t^M)$).

However, by Proposition \ref{prop:A_tilde_A}, we can find $\ell$ such that, denoting $(A^{(2)}, \Phi^{(2)}) = \SYMH(C,\tilde x)$ as in \eqref{eq:A_bar_A_def} and $s=t^\beta$,
\begin{equ}[eq:tilde_A_bar_A_far]
|\E  W_\ell [\mcF_s (\tilde A_t)] - \E W_\ell [\mcF_s ( A^{(2)}_t)] | \geq \sigma t^{1+r}\;.
\end{equ}
The conclusion \eqref{eq:W_ell_F_s_A_diff} follows from combining \eqref{eq:A_tilde_A_close} and \eqref{eq:tilde_A_bar_A_far}. 
\end{proof}

We conclude this subsection with some open problems.
\begin{enumerate}
\item
Theorem~\ref{theo:main} states that the gauge covariant renormalisation 
$\mathring C_\A$ is unique, i.e. $\mathring C_\A = \check{C}$, for each fixed 
$\mathring C_{\Phi}$. (Remark that by \cite[Remark~5.6]{CCHS_3D},
 $\check{C}$ does not depend on $\mathring C_{\Phi}$.)
If we allow $\mathring C_{\Phi}$ to vary, one has a family of gauge covariant solutions parametrised by $\mathring C_{\Phi}$.
As in Remark~\ref{rem:projection},
each of these gauge covariant solutions produces a Markov process on $\state/{\sim}$ by projection.
It would be natural to expect a `separation of Higgs mass' result, namely, the projected Markov processes on $\state/{\sim}$ for different choices of $\mathring C_{\Phi}$ are distinct. 
We expect that 
a certain type of gauge invariant observable (other than Wilson loop) consisting of the Higgs field will be useful to separate these projected Markov processes.
\item
A more general question is to `classify' all the gauge covariant dynamics. 
For instance, consider the pure YM dynamic, i.e. $\higgsvec = \{0\}$.
Theorem~\ref{theo:main} states that any finite shift of the term $\check{C} A$  yields a non-gauge covariant dynamic. 
Can one find  other gauge covariant dynamics by  finite perturbations of the SPDE, or prove non-gauge covariance for such perturbations? 
We expect that such gauge covariant perturbations are rare but exist, 
e.g. perturbing the stochastic YM equation by the gradient of the Chern--Simons functional in 3D may still be gauge covariant.
\item
As already mentioned above,
deriving the scaling limit of the 3D lattice YM or YMH Langevin dynamic is currently an open problem.
In 2D, the pure YM case was solved in \cite{chevyrev2023invariant},
which also established universality  of a large class of lattice YM models (including in particular the models with Wilson, Villain, Manton actions),
namely, they all scale to the continuum 2D YM measure, as a consequence of uniqueness of gauge covariant renormalisation proven in \cite{chevyrev2023invariant}. 
For pure YM in 3D,  we expect that the result of this paper (which obviously holds without Higgs by taking $\higgsvec=\{0\}$) may similarly help to establish uniqueness of scaling limit.
It will be also interesting
to couple with Higgs, even in 2D, and study whether the scaling limits of lattice models with various YM actions  and representations  will be identical or parametrised by a Higgs mass.
\end{enumerate}

\subsection{Idea of proof}
\label{sec:idea}

In the rest of the paper, we prove Proposition~\ref{prop:A_tilde_A}.
The main idea is to obtain a suitable expansion of $\E W_\ell[\CF_s(A_t)] - \E W_\ell[\CF_s(\tilde A_t)]$ that contains a finite number of explicit terms plus a small remainder.
The key challenge is to find `good terms' in this expansion which, after taking suitable initial conditions $\tilde x$ and $g(0)$ and $s>0$, exhibit a lower bound for all $t>0$ small.
The choice of initial conditions is rather delicate as we need to take $\tilde x$ small to ensure smallness of the remainders, but at the same time we need to use $\tilde x$ in the `good terms', so we cannot take it too small. It turns out that there is a suitable choice with $|\tilde x|_{\CC^3}\asymp t^r$.

The choice of regularisation scale $s>0$ is also 
non-trivial because we need to take $s$ sufficiently small so that we have good expansion for $\CF_s(A_t) - \CF_s(\tilde A_t)$, but at the same time $s$ should be not too small as otherwise $\CF_s(A_t)$ and $\CF_s(\tilde A_t)$ become too irregular and we lose control on the difference of Wilson loops $W_\ell[\CF_s(A)]-W_\ell[\CF_s(\tilde A)]$.
It turns out that $s=t^\beta$ for small, but not too small, $\beta>0$, satisfies both of these requirements. 

In more detail, we carry out the following steps.
\begin{enumerate}
\item
We first obtain an expansion $X_t$ and $\tilde X_t$ for small $t>0$ in Section \ref{sec:SPDE}.
Since their equations differ by the term $c\mrd \tilde g \tilde g^{-1}$ in \eqref{eq:tilde_A_equ},
we in turn obtain a short-time expansion of the difference $A_t-\tilde A_t$ with terms explicitly depending on $c$.
We note that, while our short-time analysis in this section
is stated only for SYMH, it is rather systematic and can be readily generalised to other SPDEs.

\item We then study the discrepancy between 
$\mcF_s ( A_t)$ and $\mcF_s (\tilde A_t)$ appearing in \eqref{e:1+r}.
For this purpose, we introduce in Section \ref{sec:Fs} a new deterministic state space $\SState$  of distributions on which the YM heat flow $\CF_s$ is well-defined and which encodes an a priori bound on the leading quadratic singularity of the form $A\d A$ in $\CF_s$.
Our space $\SState$ is related to but different from the state spaces in \cite{Sourav_flow,CCHS_3D};
here $\SState$ imposes finer control on the leading singularity 
by measuring regularity in the distributional Banach spaces from \cite{Chevyrev19YM,CCHS_2D} (vs. classical H\"older--Besov spaces), which are defined in terms of line integrals.
This turns out natural and important because line integrals appear in the expansion of the Wilson loop $W_\ell$.

\item
Next, in Section~\ref{sec:Wilson}, we study the discrepancy between 
$W_\ell[\mcF_s ( A_t)] $ and  $W_\ell [\mcF_s (\tilde A_t)]$  in \eqref{e:1+r}
in terms of powers of $s,t$ and the size of $A_t$ and $\tilde A_t$ in the space $\SState$.
The results of Sections \ref{sec:SPDE} - \ref{sec:Wilson} are entirely deterministic.

\item
In Section~\ref{sec:quadratic}, we show that perturbations of the 3D stochastic heat equation take values in this new space $\SState$.
We show this by a Kolmogorov-type argument which is similar but somewhat simpler than related arguments in \cite{CCHS_2D,CCHS_3D}.

\item In Section~\ref{sec:expectation-loop}, we study the discrepancy between the expectations
$\E W_\ell[\mcF_s ( A_t)] $ and  $\E W_\ell [\mcF_s (\tilde A_t)]$.
Here we make the precise choice of $s=t^\beta$.

\item
Finally, in Section \ref{sec:initial}, we choose suitable initial conditions $\tilde x$ and $g(0)$ to demonstrate the lower bound in Proposition~\ref{prop:A_tilde_A}.
Similar to \cite{chevyrev2023invariant}, this choice of initial conditions relies on the Chow--Rashevskii theorem from sub-Riemannian geometry (see Lemma \ref{lem:curve_selection}).

\end{enumerate}

\subsection{Notation}
\label{sec:notation}

We identity the torus $\T^3 = \R^3/\Z^3$, as a set, with $[0,1)^3$.
We also identify functions on $\T^3$ with periodic functions on $\R^3$.
We write $P=(P_t)_{t>0}$\label{page:P}
for the heat semigroup on $\T^3$.
We will write $\CP$ \label{page:CP}
for the  integration operator 
on modelled distributions 
associated to the heat kernel.

For $t\in [0,T]$ and a time-dependent distribution $f\colon [0,T]\to\CD'(\T^3)$, we write
\begin{equ}[eq:P_star_def]
\CP_t \star f = \int_0^t P_{t-s} f_s\mrd s
\end{equ}
whenever the integral is well-defined.
(Note that we view $\CP_t \star f $ as an entire notation, where $\CP_t$ does not have a separate meaning and should not be confused with the operator $\CP$ on modelled distributions.)
We also extend this notation to $t<0$ by $\CP_t \star f = 0$.

We let $\CC(X,Y)$ be the space of continuous functions $f\colon X\to Y$.
We write $|\cdot |_\infty = |\cdot |_{L^\infty}$
for the $L^\infty$  (extended)  norm on $\CD'(\T^3)$.
For a normed space $F$ and $\beta\leq 0$, and $f\in\CD'(\T^3,F)$, define the (extended) norm
\begin{equ}[e:def-Cbeta]
|f|_{\CC^\beta} = \sup_{s\in (0,1)} s^{-\beta/2}|P_s f|_{\infty}\;.
\end{equ}

Let $\floor\beta\in\Z$ denote the floor of $\beta\in\R$.
For $\beta>0$, we let $\CC^\beta(\R^3)$ denote the space of $\floor{\beta}$-times differentiable functions $f\colon \R^3 \to \R$ for which
\begin{equ}
|f|_{\CC^\beta(\R^3)}
\eqdef
\max_{|k|<\floor\beta} |\partial^k f |_\infty + \max_{|k|=\floor{\beta}} |\partial^{k} f |_{\CC^{\beta-\floor{\beta}}}
< \infty \;,
\end{equ}
where, for $\eta\in [0,1)$,
\begin{equ}
|f|_{\CC^\eta} = \sup_{x\neq y} |x-y|^{-\eta}|f(x)-f(y)|\;,
\end{equ}
and where $\partial^k f$ is the usual $k$-th derivative of $f$ for a multi-index $k\in \N^3$, $\N = \{0,1,\ldots\}$,
and where we denote $|k| = \sum_{i=1}^d k_i$.
We correspondingly write $\CC^\beta(\T^3)$ for the space of periodic functions of the given regularity.

For $\K\subset \R\times\R^3$, we make the same definition for $\CC^\beta(\K)$ except that, for a multi-index $k = (k_0,\ldots, k_3)\subset \N^{1+3}$, we use the parabolic scaling $|k|_\s = 2k_0 + \sum_{i=1}^3 k_i$.

Denote\label{page:O}
\begin{equ}
O = [-1,2]\times\T^3 \;.
\end{equ}
For $\beta<0$, we write $\CC^\beta(O) \subset \CD'(\R\times\T^3)$
for the space of distributions on $\R\times\T^3$ with finite
(inhomogeneous) H\"older--Besov norm
\begin{equ}
	|\xi|_{\CC^\beta(O)} = 
	\sup_{z\in O}\sup_{\phi \in \CB^r}
	\sup_{\lambda\in (0,1]} \lambda^{-\beta}|\scal{\xi,\phi^\lambda_z}|\;,
\end{equ}
where $r=-\floor{\beta}+1$, $\CB^r$ is the set of all $\phi \in \CC^\infty(\R\times\R^3)$ with support in the ball $\{z:|z|<\frac14\}$ and $\|\phi\|_{\CC^r(\R\times\R^3)}\leq 1$,
and where $\phi^\lambda_z\in\CC^\infty(\R\times\T^3)$ is given by 
\begin{equ}
\phi^\lambda_{(s,y)} (t,x) = \lambda^{-5}\phi((t-s)\lambda^{-2}, (x-y)\lambda^{-1})\;.
\end{equ}
For $\beta\in\R$, if no domain is specified, we let $\CC^\beta$ denote $\CC^{\beta}(\T^3)$.

We let $\bone_+ \colon \R\times\T^3 \to\{0,1\}$ be the indicator of the set $\{(t,x)\,:\,t>0\}$. 

For a Banach space $F$, we let $O_F(t)$ denote an element $X\in F$
such that $\|X\|_F \leq C t$ for a proportionality constant $C>0$.

For a vector space $F$ of $\mfg^3$-valued distributions, we let $F[\mfg,\mfg]$ denote the subset of those $f=(f_1,f_2,f_3)\in F$ for which $f_i$ takes values in the derived Lie algebra $[\mfg,\mfg]$ for $i=1,2,3$.

\medskip
\noindent
{\textbf{Acknowledgements}.}
We thank the anonymous referees for their careful reading of the manuscript and valuable comments.
IC gratefully acknowledges support from the DFG CRC/TRR 388 `Rough
Analysis, Stochastic Dynamics and Related Fields' through a Mercator Fellowship held at TU Berlin
and from the ERC via the grant SQGT 101116964. HS gratefully acknowledges supports by NSF through CAREER DMS-2044415, 
and by the Simons Foundation through a Simons Fellowship.

\section{Short-time estimates for SPDEs}
\label{sec:SPDE}

In this section, we perform a short-time analysis of the equations \eqref{eq:tilde_A_equ}-\eqref{eq:PDE_g} and keep track of how different choices for $c$ affect the solution.
Our first step is to write the stochastic processes from Proposition \ref{prop:A_tilde_A} in the more general form
\begin{equs}[eq:X_h_new]
\partial_t  X
&= \Delta X +  X\partial  X +  X^3 +\moll^\eps* \xi + C^\eps  X + (ch,0)\;,
\\
\partial_t  h_{i} &=   \Delta h_{i}
 - [h_j,\partial_j h_i] + [[A_j, h_j],h_i] + \partial_i [A_j, h_j]\;,
\end{equs}
where $c\in L(\mfg^3)$ (possibly zero).
(Taking $c=0$ recovers $X$ from Proposition \ref{prop:A_tilde_A}, and taking $c$ with $c_1\neq 0$ recovers $\tilde X$
with $h=(\mrd \tilde g)\tilde g^{-1}$.)
We refer to \cite[Lem.~7.2]{CCHS_2D} or \cite[Lem.~6.3]{CCHS_3D} for the derivation of the second equation in
\eqref{eq:X_h_new} from  \eqref{eq:tilde_A_equ}-\eqref{eq:PDE_g}.

\subsection{Lifting to the space of modelled distributions}
Let us fix 
\begin{equ}[e:fix-omega]
\omega\in(-1/2,0) \;, \qquad 
\kappa\eqdef \tfrac{1}{100}(\omega+1/2) \wedge\tfrac{1}{100}(-\omega)  \in (0,\tfrac{1}{200})\;.
\end{equ}
We use spaces of singular modelled distributions $\cD^{\gamma,\eta}_\alpha$ from \cite[Sec.~6]{Hairer14}.
All models and modelled distributions are considered on the set $O = [-1,2]\times\T^3$ or a subset thereof.
Without further mention, we will frequently use the multiplication bound for $f_1,f_2$ which can be multiplied
\begin{equ}
|f_1 f_2|_{\cD^{\gamma,\eta}_\alpha} \lesssim |f_1|_{\cD^{\gamma_1,\eta_1}_{\alpha_1}}
|f_2|_{\cD^{\gamma_2,\eta_2}_{\alpha_2}}
\end{equ}
where $\gamma= (\gamma_1+\alpha_2)\wedge(\gamma_2+\alpha_1)$,
$\eta= (\eta_1+\eta_2)\wedge (\eta_1+\alpha_2)\wedge(\eta_2+\alpha_1)$,
and $\alpha = \alpha_1\wedge\alpha_2$,
together with the differentiation bound
\begin{equ}
|\d f|_{\cD^{\gamma-1,\eta-1}_{\alpha-1}} \lesssim |f|_{\cD^{\gamma,\eta}_{\alpha}}\;.
\end{equ}

We next write \eqref{eq:X_h_new} on the level of modelled distributions as
\begin{equs}
\mcX &=P X(0) + \bPsi +  \CP \bone_+ \big\{ \mcX\partial \mcX +  \mcX^3+\mathring C\mcX + c\mcH \big\}
\label{eq:X-formal}
\\
&\eqdef
P X(0) + \bPsi + \CP \bone_+ \big\{Q^\ymh(\mcX)  + c\mcH \big\}\;,
\\
\mcH &= P h(0) + \CP  \bone_+ (\mcH\partial \mcH + \mcX \mcH^2)
 + \CP' \bone_+ (\mcX \mcH)\;,\label{eq:H_formal}
\end{equs}
where the initial conditions $X(0),h(0)$ are smooth and $PX(0)$, $P h(0)$ are interpreted as modelled distributions in $\cD^{\infty,\infty}_0$ valued in the polynomial regularity structure,
$\CP'$ denotes the integration operator associated to the spatial derivative of the heat kernel,
and
\begin{equ}[eq:tilde_Psi]
\bPsi 
=  \CP^{\bone_+\xi}\bone_+\Xi\in \cD^{\frac32+2\kappa,-\frac12-\kappa}_{-\frac12-\kappa} \;.
\end{equ}
Here $\CP^{\bone_+\xi}$ is the integration map with an `input' distribution $\bone_+\xi$
that is compatible with $\bone_+\Xi \in \cD^{\infty,\infty}_{-\frac52-\kappa}$,
see \cite[Appendix~A]{CCHS_3D} for the definition of such integration maps
and the notion of compatibility (basically, it means that, in $\CP^w f$, the reconstruction of $f\in \cD^{\gamma,\eta}_\alpha$ and $w\in \CC^{\eta\wedge \alpha}(O)$ coincide on $O$ away from $t=0$.) 
We recall that the integration map with `inputs' is necessary here
since the normal integration map $\CP$ in \cite{Hairer14} contains a reconstruction operator
which only applies to elements of $\cD^{\gamma,\eta}_\alpha$ with $\eta\wedge \alpha >-2$. 
For now, $\mathring C \in L(E)$ is any linear map.
We also use $Q^\ymh(\CX)$ as shorthand for the polynomial $\CX\d\CX + \CX^3 + \mathring C \CX$.
We have furthermore fixed a regularity structure as in \cite[Sec.~5]{CCHS_3D}, with the obvious modifications to handle the component $\CH$,
as well as a model $Z$.

We would like to solve for $(\CX,\CH)$ in $\cD^{\frac32+2\kappa,-\frac12-\kappa}_{-\frac12-\kappa} \times \cD^{1+2\kappa,0}_{0}$.
However, as explained in \cite[Sec.~5.2]{CCHS_3D}, the product $\CX\d\CX$ creates a non-integrable singularity at time $t=0$ due to the exponent $-\frac12-\kappa$,
so to solve \eqref{eq:X-formal},
we decompose
\begin{equ}
\mcX=\mcY+ \bPsi\;,
\end{equ}
where now $\mcY$ solves
\begin{equs}[eq:tilde_X]
\mcY &= PX(0)+ \CP \bone_+ \big\{\mcX^3 +c \mcH + \mathring C \mcX
\big\}
+ \CP^{\Psi\partial\Psi} (\bPsi\partial\bPsi) \\
&\qquad\qquad
+\CP \bone_+ (\mcY \partial \bPsi + \bPsi \partial \mcY + \mcY \partial \mcY)
\\
&= PX(0)+ \CP \mathbf{1}_+ \big\{
\tilde Q^\ymh(\mcY)+c \mcH 
\big\} + \CP^{\Psi\partial\Psi} (\bPsi\partial\bPsi)  \;,
\end{equs}
where $\Psi\partial\Psi \in \CC^{-2-2\kappa}(O)$ is an input distribution that is compatible with $\bPsi\d\bPsi \in \cD^{\kappa,-2-2\kappa}_{-2-2\kappa}$
and where $\CP^{\Psi\partial\Psi} (\bPsi\partial\bPsi) \in \cD^{2+\kappa,-2\kappa}_{-2\kappa}$
--- see \cite[Lem.~5.18]{CCHS_3D} (and \eqref{eq:Psi_d_Psi_def} below)
where such a  compatible distribution is constructed probabilistically for a white noise $\xi$.
Above, to simplify notation, we write $\tilde Q^\ymh(\mcY)$
for the polynomial
\[
\tilde Q^\ymh(\mcY) = 
\mcX^3 + \mathring C \mcX 
+\mcY\partial \bPsi + \bPsi \partial \mcY + \mcY \partial \mcY\;.
\]
One has
\begin{equ}\label{eq:H_map}
\tilde Q^\ymh\colon \cD^{\frac32+2\kappa, \omega}_{-\kappa} \to 
\cD^{\kappa,\omega-\frac32-\kappa}_{-\frac32-2\kappa} 
\quad \text{is locally Lipschitz}\;,
\end{equ}
where the worst term is $\mcY \partial \bPsi$, for which one has 
\[
\cD^{\frac32+2\kappa,\omega}_{-\kappa} \times \cD^{\frac12+2\kappa,-\frac32-\kappa}_{-\frac32-\kappa}
\to \cD^{\kappa,\omega-\frac32-\kappa}_{-\frac32-2\kappa} \;.
\]
Remark that 
 $\omega-\frac32- \kappa> -2$ by our choice \eqref{e:fix-omega} which is important for applying the operator $\CP$ in the following.

Standard arguments imply that we can solve for
\begin{equ}
(\mcY,\CH) \in \cD^{\frac32+2\kappa,\omega}_{-\kappa}\times \cD^{1+2\kappa,0}_{0}\;,
\end{equ}
where we recall $\omega$ from \eqref{e:fix-omega}
(this is similar and even simpler than \cite[Sec.~5.2]{CCHS_3D}, where singularity of the initial condition $X(0)$ requires a further decomposition of $\mcY$).

In this way, we \emph{define} the solution $(\CX,\CH) \in \cD^{\frac32+2\kappa,-\frac12-\kappa}_{-\frac12-\kappa}\times \cD^{1+2\kappa,0}_{0}$ to \eqref{eq:X-formal}-\eqref{eq:H_formal}
as $\CX = \mcY + \bPsi$ where $(\mcY,\CH)$ solves \eqref{eq:tilde_X}--\eqref{eq:H_formal}.

\subsection{Short time expansion}

We now proceed to the short-time analysis of \eqref{eq:X-formal}-\eqref{eq:H_formal}.
Recall $O = [-1,2]\times \T^3$.
For $t\in (0,1]$, we use the shorthand $|\cdot|_{\cD^{\gamma,\eta}}\eqdef |\cdot|_{\cD^{\gamma,\eta};(0,t]\times\T^3}$.

We will frequently apply the following:
Let $\theta\geq0$, $\gamma>0$, $\alpha\leq 0$ and $\eta\in\R$. Set $\bar\eta=(\eta\wedge\alpha)+2-\theta$.
If $\alpha\wedge\eta>-2$, then for $t\in(0,1]$ and $f\in\cD^{\gamma,\eta}_\alpha$
\begin{equ}\label{eq:short_time_conv}
|\CP\bone_+f|_{\cD^{\gamma+2,\bar\eta}} \lesssim t^{\theta/2}|f|_{\cD^{\gamma,\eta}}\;,
\quad
|\CP'\bone_+f|_{\cD^{\gamma+1,\bar\eta-1}} \lesssim t^{\theta/2}|f|_{\cD^{\gamma,\eta}}\;,
\end{equ}
where the proportionality constant depends on the Greek letters and affinely on $\$Z\$_{\gamma;O}$,
where $Z$ is the model we fixed above.
See \cite[Thm.~7.1]{Hairer14} for the above result and \cite[(2.16)]{Hairer14} for the notation $\$Z\$_{\gamma;O}$.


Moreover, suppose $\eta\wedge\alpha+2 > -2$ and consider $w\in \CC^{\eta\wedge \alpha}(O)$  compatible with $f\in\cD^{\gamma,\eta}_\alpha$. Then (see \cite[Lem.~5.2]{GH19_boundaries})
\begin{equ}\label{eq:short_time_conv1}
|\CP^w\bone_+f|_{\cD^{\gamma+2,\bar\eta}} \lesssim t^{\theta/2} (|f|_{\cD^{\gamma,\eta}}+|w |_{ \CC^{\eta\wedge \alpha}(O)})
\end{equ}
(in contrast to \eqref{eq:short_time_conv},
\eqref{eq:short_time_conv1} does not require $\eta\wedge\alpha>-2$).

Throughout this section, we use the following notation.

\begin{notation}\label{not:t}
We take $\tau \in (0,\frac12)$ small such that
\begin{equs}[eq:size]
\tau^{-1/q} \lesssim 2
&+ \$Z\$_{\frac32+2\kappa;O} + |X(0)|_{\CC^3} + |h(0)|_{\CC^3}
\\
&+ |\CP\star \bone_+\xi|_{\CC([-1,3],\CC^{-1/2-\kappa})}
+ |\CP\star (\Psi\d\Psi)|_{\CC([-1,3],\CC^{-2\kappa})}
\end{equs}
where $q\geq 1$ is sufficiently large to ensure the existence of solutions $\CX,\CH$ to \eqref{eq:X-formal}-\eqref{eq:H_formal} on $(0,\tau)$.
We let $t \in (0,\tau)$.
Furthermore, all implicit proportionality constants are sufficiently large powers of the right-hand side of \eqref{eq:size}
that are uniform in $t\in (0,\tau)$.
\end{notation}

Remark that $|\bone_+ f|_{\CC^{\eta-2}(O)}\lesssim |\CP\star f|_{\CC([-1,3],\CC^{\eta})}$ due to $|(\partial_t-\Delta)u|_{\CC^{\eta-2}}
\lesssim |u|_{\CC^{\eta}}$,
hence the final two terms on the right-hand side of \eqref{eq:size} bound the terms of the form $|w |_{ \CC^{\eta\wedge \alpha}(O)}$ in \eqref{eq:short_time_conv1} for input distributions $w$ appearing in \eqref{eq:tilde_Psi} and \eqref{eq:tilde_X}.

We let $\mcY,\CH$ solve \eqref{eq:tilde_X} and \eqref{eq:H_formal}.
The next two lemmas give the first step
in a perturbative estimate for $\mcY,\CH$.

\begin{lemma}
\label{lem:Euler_h1}
$
\mcH = P h(0) +  O_{\cD^{1+2\kappa,-\kappa}_0}(t^{1/4})
$.
\end{lemma}

\begin{remark}\label{rem:h_general}
This short time expansion of $\CH$ is the only property of the equation \eqref{eq:PDE_g} and the $h$-component of \eqref{eq:X_h_new} that we use.
\end{remark}

\begin{proof}
By \eqref{eq:short_time_conv}
\begin{equs}
|\CP' \bone_+ (\mcX\mcH)|_{\cD^{1+2\kappa,-\kappa}_0} 
&\lesssim
 t^{1/4}
|\mcX\mcH|_{\cD^{\frac12+\kappa,-\frac12-\kappa}_{-\frac12-\kappa}}
\\
&\lesssim 
t^{1/4}
|\mcX|_{\cD^{\frac32+2\kappa,-\frac12-\kappa}_{-\frac12-\kappa}}
|\mcH|_{\cD^{1+2\kappa,0}_{0}}
\lesssim t^{1/4}\;.
\end{equs}
The other terms yield higher powers of $t$:
\begin{equ}
|\CP  \bone_+ (\mcH\partial \mcH )|_{\cD^{1+2\kappa,0}_0 } 
\lesssim t\;,
\qquad
|\CP  \bone_+ ( \mcX\mcH^2)|_{\cD^{\frac32+\kappa,-\kappa}_0 } 
\lesssim t^{3/4}\;.
\end{equ}
Combining the above bounds completes the proof.
\end{proof}

\begin{lemma}
\label{lem:Euler_A1}
$
\mcY = 
PX(0)
+O_{\cD^{\frac32+2\kappa,\omega}_{-\kappa}}(t^{-\omega/2-\kappa/2})
$.
\end{lemma}

\begin{proof}
By~\eqref{eq:H_map},~\eqref{eq:short_time_conv}, 
and recalling that $\omega-\frac32- \kappa> -2$, we obtain
\[
|\CP\bone_+ ( \tilde Q^\ymh(\mcY))|_{\cD^{\frac32+2\kappa,\omega}_0}
 \lesssim 
 t^{1/4-\kappa/2}
|\tilde Q^\ymh(\mcY)|_{\cD^{\kappa,\omega-\frac32- \kappa}_{-\frac32-2\kappa} }
 \lesssim t^{1/4-\kappa/2}\;.
\]
Also, since $\Psi\partial\Psi \in \CC^{-2-\kappa}(O)$, by \eqref{eq:short_time_conv1}
we get the following bound (which is slightly worse than the previous one since $\omega>-1/2$):
\[
|\CP^{\Psi\partial\Psi} (\bPsi\partial\bPsi) |_{\cD^{\frac32+2\kappa,\omega}_{-\kappa}}
\lesssim
 t^{-(\kappa+\omega)/2}\;.
\]
Finally, by Lemma \ref{lem:Euler_h1},
we have a better bound for the term $\CP\mcH$:
\begin{equs}
|\CP \mcH|_{\cD^{\frac32+2\kappa,\omega}_{-\kappa}} 
&\lesssim 
|\CP P h(0)|_{\cD^{\frac32+2\kappa,\omega}_{-\kappa}} 
+ |\CP O_{\cD^{1+2\kappa,-\kappa}_0 }(t^{1/4})|_{\cD^{\frac32+2\kappa,\omega}_{-\kappa}} 
\\
&\lesssim 
t^{1-\frac{\omega}2} |P h(0)|_{\cD^{0+,0}_0}
+ t^{1/4+1-\frac\kappa2 -\frac{\omega}{2}}
\lesssim t^{1-\omega/2}\;.
\end{equs}
\end{proof}

We now iterate the above procedure.
Define 
\begin{equ}
B_0 = P X(0)+\bPsi \;,
\qquad
h_0= P h(0)\;.
\end{equ}
Then, for $n\geq 1$, we define $B_n$ recursively by
\begin{equ}[e:Bn-iter]
B_n = \CP(\mathring CB_{n-4}{\bf 1}_{n\ge 4})
+ \!\!\!\!\sum_{k_1+k_2=n-1} \!\!\!\! \CP(B_{k_1}\partial B_{k_2})
+ \!\!\!\!\!\! \sum_{k_1+k_2+k_3=n-2} \!\!\!\!\!\!\!\! \CP(B_{k_1} B_{k_2} B_{k_3})
\end{equ}
where the summation indices are over integers $k_i\ge 0$
and where the term $\CP(\bPsi\partial\bPsi)$ that arises in the case $n=1$ from $\CP (B_0 \d B_0)$ is understood as $\CP^{\Psi\partial\Psi} (\bPsi\partial\bPsi)$.

Finally, we define $q_0$ 
and $r_n$ for $n=0,\ldots,5$ by
\begin{equ}[eq:def-rn]
\mcX=\sum_{i=0}^n  B_i 
+c \,{\bf 1}_{n=5} \CP h_0
+ r_n\;,
\qquad
\mcH= h_0  + q_0 \;.
\end{equ}
In view of Lemma~\ref{lem:Euler_A1} and Lemma~\ref{lem:Euler_h1},
\begin{equ}[e:r0r0]
r_0=O_{\cD^{\frac32+2\kappa,\omega}_{-\kappa}}(t^{-\omega/2-\kappa/2})\;,
\qquad
q_0 =  O_{\cD^{1+2\kappa,-\kappa}_0 }(t^{1/4})\;.
\end{equ}

\begin{remark}
While expansions to level $n\leq 5$ suffice for us,
it is possible to systematise the construction to obtain 
expansions of $\CX$ and $\CH$ to arbitrary order.
\end{remark}

\begin{remark}
As an example of \eqref{e:Bn-iter}, we have $B_1=\CP(B_0\partial B_0) = \CP^{\Psi\partial\Psi} (\bPsi\partial\bPsi)$
and $B_2 = \CP(B_0\d B_1)+\CP(B_1\d B_0) + \CP (B_0^3)$.
The motivation of our definition \eqref{e:Bn-iter} is that, for each $n$, the terms on the right-hand side are `homogeneous', i.e. when measured 
in a suitable space of modelled distributions, they
vanish as the {\it same} power of $t$ (up to a multiple of $\kappa$),
see the proof of Proposition~\ref{prop:SPDE-Euler} below.\footnote{The notation in  \cite[Section~8]{chevyrev2023invariant} is slightly different,
e.g. $B_1$ therein also contains the cubic term $B_0^3$.  This is because the expansion therein was rather low order whereas here it is more important to organise the terms in a systematic way.}
%
\end{remark}

Note that the recursive definition  \eqref{e:Bn-iter} is formal so far because we have not specified the modelled distribution spaces to which $B_n$ belong to.
The following result makes precise the spaces that $B_n$ belong to, as well as their small time asymptotics, in particular showing that 
the right-hand side of \eqref{e:Bn-iter} is well-defined.

\begin{proposition}
\label{prop:SPDE-Euler}
Define
\begin{equs}
\eta(0)=-1/2-\kappa\;,\qquad & \eta(n)=-1/2+2\kappa \qquad  (1\leq n \leq 5)\;,
\\
b(0)=0\;,\qquad  & b(n)=(1/4-\kappa/2)n-3\kappa/2 \qquad  (1\leq n \leq 5)
\end{equs}
and $\alpha(0)=-\frac12-\kappa$, $\alpha(1)=-2\kappa$, $\alpha(n)=0$ for $n\ge 2$.
Then
\begin{equ}
|B_n|_{\cD^{\frac32+2\kappa,\eta(n)}_{\alpha(n)}}
 \lesssim t^{b(n)}
 \qquad \forall 0\le n \le 5\;.
\end{equ}
Moreover,
\begin{equ} 
r_0=O_{\cD^{\frac32+2\kappa,\omega}_{-\kappa}}(t^{-\omega/2-\kappa/2})\;,
\qquad
r_n = O_{\cD^{\frac32+2\kappa,\omega}_{0}}(t^{(n+1)/4-\kappa_n})
	 \qquad \forall 1\le n \le  5\;,
\end{equ}
where $\kappa_n \eqdef \frac12 (\omega+\frac12) + (1+\frac{n}2)\kappa>0$.
\end{proposition}

\begin{proof}
We proceed by induction.
For the base case $n=0$, the claimed bounds on $B_0$ and $r_0$ are simply \eqref{eq:tilde_Psi} and \eqref{e:r0r0} respectively so the claims are true.
Consider now $n\ge 1$
and suppose that the claim holds for $B_k$ and $r_k$ for $k< n$.

For the following calculation it is useful to note that $\eta(n)\leq\alpha(n)$ and
\[
\frac12\eta(n)+b(n) = -\frac14-\frac{\kappa}2 +\Big(\frac14-\frac{\kappa}2\Big)n \qquad \forall n\ge 0\;.
\]
By the induction hypothesis and \eqref{eq:short_time_conv}-\eqref{eq:short_time_conv1},  for $k_1+k_2=n-1$, 
\begin{equs}[e:Bk1dBk2]
| &\CP(B_{k_1}\partial  B_{k_2}) |_{\cD^{\frac32+2\kappa,-\frac12+2\kappa}_{\alpha(n)}}
\\
&\lesssim
t^{\frac12 (\eta(k_1)+\eta(k_2)+\frac32-2\kappa)} 
|B_{k_1}\partial B_{k_2}|_{\cD^{\kappa,\eta(k_1)+\eta(k_2)-1}_{\alpha(k_1)+\alpha(k_2)-1}}
\\
&\lesssim 
t^{\frac12 (\eta(k_1)+\eta(k_2)+\frac32-2\kappa)} 
|B_{k_1}|_{\cD^{\frac32+2\kappa,\eta(k_1)}_{\alpha(k_1)}}
|\partial B_{k_2}|_{\cD^{\frac12+2\kappa,\eta(k_2)-1}_{\alpha(k_2)-1}}
\\
&\lesssim 
t^{\frac12 (\eta(k_1)+\eta(k_2)+\frac32-2\kappa)} \cdot
t^{b(k_1)} \cdot t^{b(k_2)}
=
t^{b(n)} \;.
\end{equs}
Here we used  $\alpha(k_1)+\alpha(k_2)+1\ge \alpha(k_1+k_2+1)=\alpha(n)$
in the first step.
Also,  in the first step, when $k_1=k_2=0$, recalling our convention below \eqref{e:Bn-iter},  we can apply \eqref{eq:short_time_conv1};
otherwise we have $\eta(k_1)+\eta(k_2)-1 > -2$  so that we can apply \eqref{eq:short_time_conv}.

Similarly, for $k_1+k_2+k_3=n-2$ where $n\ge 2$, 
\begin{equs}
|\CP( B_{k_1}  B_{k_2} B_{k_3} & ) |_{\cD^{\frac32+2\kappa,-\frac12+2\kappa}_{\alpha(n)}}
\lesssim
 t^{\frac12(\sum_i \eta(k_i) + \frac52 -2\kappa)} 
\prod_{i=1}^3 |B_{k_i}|_{\cD^{\frac32+2\kappa,\eta(k_i)}_{\alpha(k_i)}}
\\
&\lesssim 
 t^{\frac12(\sum_i \eta(k_i) + \frac52 -2\kappa)} \cdot
t^{\sum_i b(k_i)}
=
t^{b(n)} \;.		\label{e:Bk1Bk2Bk3}
\end{equs}
Moreover, for $n\ge 4$,
\begin{equs}
|\CP(B_{n-4}) |_{\cD^{\frac32+2\kappa,-\frac12+2\kappa}_{\alpha(n)}}
&\lesssim
t^{\frac12(\eta(n-4)+\frac52-2\kappa)}
 |B_{n-4}|_{\cD^{\frac32+2\kappa,\eta(n-4)}_{\alpha(n-4)}}
\\
&\lesssim 
t^{(\frac{1}{4}-\frac\kappa2)n + \frac{\kappa}{2}}
\leq
t^{b(n)}\;.
\end{equs}
Therefore, by \eqref{e:Bn-iter}, we have the desired bound on $B_n$.

Turning to the remainders $r_n$ for $1\le n \le 5$,
note that, by our equation \eqref{eq:X-formal} and definition \eqref{eq:def-rn},
\[
\sum_{i=0}^n  B_i 
+c \,{\bf 1}_{n=5} \CP h_0 + r_n 
=
B_0 
 + \CP \bone_+ \big\{Q^\ymh(\mcX) + c\mcH \big\}\;,
\]
where we recall $Q^\ymh(\CX)=
\CX\d\CX + \CX^3 + \mathring C \CX$, so using  \eqref{eq:def-rn} again in the above identity one has
\begin{equ}
r_n = \CP \bone_+ \Big\{
 Q^\ymh \Big(\sum_{i=0}^{n-1}  B_i + r_{n-1}\Big)
 +c \Big(h_0+q_0 - {\bf 1}_{n=5}h_0 \Big)
 \Big\}
-\sum_{i=1}^n B_i  \;.
\end{equ}
By  \eqref{e:Bn-iter},  the last term $\sum_{i=1}^n B_i $ cancels 
a large number of terms from the expansion of $Q^\ymh$.
One then has, for $n\ge 1$,
\begin{equs}[e:r_n-equ]
r_n &=\CP \bone_+ \Big(
\sum_{k_1+k_2\ge n} B_{k_1} \partial B_{k_2}
+
\sum_{k_1+k_2+k_3\ge n-1} B_{k_1} B_{k_2} B_{k_3}
\\
&\quad+r_{n-1} \partial r_{n-1}
+ \Big(\sum_{i=0}^{n-1} B_i\Big) \partial r_{n-1}
+r_{n-1}\partial \Big(\sum_{i=0}^{n-1} B_i\Big)
\\
&\quad+ r_{n-1}^3
+ 3 r_{n-1}^2 \Big(\sum_{i=0}^{n-1} B_i\Big)
+ 3 r_{n-1} \Big(\sum_{i=0}^{n-1} B_i\Big)^2
+c \Big(h_0+q_0 - {\bf 1}_{n=5}h_0 \Big)
\\
&\quad +\mathring{C} \Big(\sum_{i=(n-3)\vee 0}^{n-1}  B_i  + r_{n-1}\Big)
 \Big)\;,
\end{equs}
where $k_1,k_2,k_3<n$.

Similarly as in \eqref{e:Bk1dBk2}, the first term is bounded as 
\begin{equs}
|\CP(B_{k_1}\partial B_{k_2}) |_{\cD^{\frac32+2\kappa,\omega}_{0}}
&\lesssim 
t^{\frac12 (\eta(k_1)+\eta(k_2)+1-\omega)} \cdot
t^{b(k_1)} \cdot t^{b(k_2)}
\\
&\leq
t^{-\omega/2-\kappa+(1/4-\kappa/2)n}
=t^{(n+1)/4-\kappa_n}
\end{equs}
for $k_1+k_2\ge n$.
For the second term, as in \eqref{e:Bk1Bk2Bk3},
\begin{equs}
|\CP(B_{k_1} B_{k_2} B_{k_3}) |_{\cD^{\frac32+2\kappa,\omega}_{0}}
&\lesssim 
 t^{\frac12(\sum_i \eta(k_i) + 2-\omega)} \cdot
t^{\sum_i b(k_i)}
\\
&\leq
t^{\frac14-\frac{\omega}2-\frac32 \kappa 
+(1/4-\kappa/2)(n-1)} 
=t^{(n+1)/4-\kappa_n}
\end{equs}
for $k_1+k_2+k_3\ge n-1$. 
Also, for $(n-3)\vee 0 \le i \le n-1$, 
\begin{equs}
|\CP  B_i |_{\cD^{\frac32+2\kappa,\omega}_{0}}
&\lesssim 
t^{\frac12 (\eta(i)+2-\omega)}
 |B_i |_{\cD^{\frac32+2\kappa,\eta(i)}_{\alpha(i)}}
\\
& \lesssim 
t^{\frac12 (\eta(i)+2-\omega)}
\cdot t^{b(i)}
=
t^{\frac{3}{4}-\frac{\omega}{2}-\frac{\kappa}{2}+(\frac14-\frac\kappa2)i}
\leq
t^{(n+1)/4-\kappa_n}\;.
\end{equs}
Regarding $\big(\sum_{i=0}^{n-1} B_i\big) \partial r_{n-1}$,
it suffices to bound the term
$B_0 \partial r_{n-1}$ since the other terms satisfy better bounds.
One has,
\begin{equs}
|\CP(B_{0}\partial r_{n-1}) |_{\cD^{\frac32+2\kappa,\omega}_{0}}
&\lesssim
t^{1/4-\kappa/2}
 |B_{0}\partial r_{n-1}|_{\cD^{\kappa,\omega-\frac32-\kappa}_{-\frac32-2\kappa}}
\\
&\lesssim 
t^{1/4-\kappa/2} 
|B_{0}|_{\cD^{\frac32+2\kappa,-\frac12-\kappa}_{-\frac12-\kappa}}
|\partial r_{n-1}|_{\cD^{\frac12+2\kappa,\omega-1}_{-1-\kappa}}
\\
&\lesssim 
t^{1/4-\kappa/2}  t^{n/4-\kappa_{n-1}} 
=
t^{(n+1)/4-\kappa_n} \;.
\end{equs}
Here, we again used $\omega-\frac32-\kappa> -2$,
so the usual Schauder estimate \eqref{eq:short_time_conv} for $\CP$ applies.
Also we get the exponent $\omega-\frac32-\kappa$ for $B_0 \partial r_{n-1}$
using $-\frac32-2\kappa > \omega-\frac32-\kappa$ by \eqref{e:fix-omega}.
Similarly, for
the term
$r_{n-1}\partial \big(\sum_{i=0}^{n-1} B_i\big)$, we consider the worst term and
\begin{equs}
|\CP( r_{n-1} \partial B_{0}) |_{\cD^{\frac32+2\kappa,\omega}_{0}}
\lesssim
t^{(n+1)/4-\kappa_n} \;.
\end{equs}

The above terms have been shown to be order at most $t^{(n+1)/4-\kappa_n}$.
We now show that the other terms in 
\eqref{e:r_n-equ} satisfy even better bounds.

The worst term in 
$r_{n-1} \big(\sum_{i=0}^{n-1} B_i\big)^2$
is 
$B_0^2 r_{n-1} $. One has,
\begin{equs}
|\CP(B_{0}^2 r_{n-1} ) |_{\cD^{\frac32+2\kappa,\omega}_{0}}
&\lesssim
t^{1/2-\kappa}
 |B_{0}^2 r_{n-1}|_{\cD^{\frac12,\omega-1-2\kappa}_{-1-3\kappa}}
\\
&\lesssim t^{1/2-\kappa} 
|B_{0}|_{\cD^{\frac32+2\kappa,-\frac12-\kappa}_{-\frac12-\kappa}}^2
| r_{n-1}|_{\cD^{\frac32+2\kappa,\omega}_{-\kappa}}
\\
&\lesssim 
t^{1/2-\kappa}  t^{n/4-\kappa_{n-1}} 
=
t^{\frac14-\frac\kappa{2}} t^{(n+1)/4-\kappa_n} 
\leq
t^{(n+1)/4-\kappa_n} \;.
\end{equs}
Also, the worst term in 
$r_{n-1}^2 \big(\sum_{i=0}^{n-1} B_i\big)$
is 
$B_0 r_{n-1}^2$. One has, 
\begin{equs}
|\CP(B_{0} r_{n-1}^2 ) |_{\cD^{\frac32+2\kappa,\omega}_{0}}
&\lesssim
t^{\frac\omega2+\frac34-\frac\kappa2} 
|B_{0} r_{n-1}^2|_{\cD^{1,2\omega-\frac12-\kappa}_{-\frac12-3\kappa}}
\\
&\lesssim 
t^{\frac\omega2+\frac34-\frac\kappa2} 
|B_{0}|_{\cD^{\frac32+2\kappa,-\frac12-\kappa}_{-\frac12-\kappa}}
| r_{n-1}|_{\cD^{\frac32+2\kappa,\omega}_{-\kappa}}^2
\\
&\lesssim 
t^{\frac\omega2+\frac34-\frac\kappa2}  
(t^{n/4-\kappa_{n-1}})^2
\leq
t^{(n+1)/4-\kappa_n} \;.
\end{equs}
Moreover,
\begin{equs}
|\CP(r_{n-1} \partial r_{n-1}) |_{\cD^{\frac32+2\kappa,\omega}_{0}}
&\lesssim
t^{(1+\omega)/2}
 |r_{n-1} \partial r_{n-1}|_{\cD^{\frac12+\kappa,2\omega-1}_{-1-2\kappa}}
\\
&\lesssim
t^{(1+\omega)/2} 
| r_{n-1}|_{\cD^{\frac32+2\kappa,\omega}_{-\kappa}}
| \partial r_{n-1}|_{\cD^{\frac12+2\kappa,\omega-1}_{-1-\kappa}}
\\
&\lesssim 
t^{(1+\omega)/2}
(t^{n/4-\kappa_{n-1}} )^2
\leq
t^{(n+1)/4-\kappa_n} \;,
\end{equs}
and
\begin{equs}
|\CP(r_{n-1}^3) |_{\cD^{\frac32+2\kappa,\omega}_{0}}
&\lesssim
t^{1+\omega}
 |r_{n-1}^3|_{\cD^{\frac32,3\omega}_{-3\kappa}}
 \lesssim
t^{1+\omega}
 |r_{n-1}|_{\cD^{\frac32+2\kappa,\omega}_{-\kappa}}^3
\\
&\lesssim 
t^{1+\omega}
(t^{n/4-\kappa_{n-1}} )^3 
\leq
t^{(n+1)/4-\kappa_n} \;.
\end{equs}
The term $\CP r_{n-1}$ is bounded analogously.
Finally, by \eqref{e:r0r0}, for $1\le n\le 5$,
\[
|\CP q_0|_{\cD^{\frac32+2\kappa,\omega}_{0}}
\lesssim
t^{\frac12 (2-\kappa-\omega)} 
|q_0 |_{\cD^{\frac12+2\kappa,-\kappa}_{0}}
\lesssim
t^{\frac12 (2-\kappa-\omega)} \cdot t^{\frac14}
\leq
t^{(n+1)/4-\kappa_n} \;,
\]
and for $1\le n\le 4$, since $h_0=P h(0)\in \cD^{\frac12+2\kappa,0}_0$
\[
|\CP h_0|_{\cD^{\frac32+2\kappa,\omega}_{0}}
\lesssim
t^{\frac12 (2-\omega)} 
|h_0 |_{\cD^{\frac12+2\kappa,0}_{0}}
\lesssim
t^{(n+1)/4-\kappa_n} \;.
\]
Note that, if $n=5$, the final bound may not hold, but ${\bf 1}_{n=5}h_0$ cancels $h_0$
  in \eqref{e:r_n-equ}, so we do not need to bound $\CP h_0$.
This proves the desired bound on $r_n$.
\end{proof}

We record a simple but useful consequence of Proposition \ref{prop:SPDE-Euler}.
Recall the notation $\CP_t \star f$ from \eqref{eq:P_star_def}.
Denote
\begin{equ}[eq:Psi]
\Psi = \CR \bPsi = \CP\star \bone_+\xi\;,
\end{equ}
which is in $\CC([-1,3],\CC^{-1/2-\kappa})$ whenever the right-hand side of \eqref{eq:size} is finite.

\begin{lemma}\label{lem:a_remainder}
Suppose $|X(0)|_{\CC^3}\lesssim 1$. Then
\begin{equ}
(\CR \mcX)(t)
= X(0) + \Psi_t + \CP_t \star (\Psi \d \Psi) + O_{L^\infty}(t^{\frac14-3\kappa/2})\;.
\end{equ}
\end{lemma}

\begin{proof}
As in \eqref{eq:def-rn}
we  have $\mcX = B_0 + B_1 + r_1$ where $\CR B_0 = PX(0)+\Psi $ and
\begin{equ}[eq:RB1]
\CR B_1 = \CP \star \Big( PX(0) P'X(0) 
+PX(0) \d  \Psi +  P'X(0) \Psi 
+ \Psi \d \Psi \Big)\;.
\end{equ}
Since $|X(0)|_{\CC^3}\lesssim 1$, we have $| \CP  \star ( PX(0) P'X(0)) |_{L^\infty} \lesssim t$.
Also,
\[
|P X(0) \d \bPsi|_{\cD^{\frac12,-\frac32-\kappa}_{-\frac32-\kappa}} \lesssim 1\;,
\qquad
|P' X(0) \bPsi|_{\cD^{\frac12,-\frac12-\kappa}_{-\frac12-\kappa}}
\lesssim 1\;,
\]
and thus $|\CP(P X(0) \d \bPsi)|_{\cD^{\frac12,0}_{0}}\lesssim t^{\frac14-\kappa/2}$ and
$|\CP(P' X(0) \bPsi)|_{\cD^{\frac12,0}_{0}}\lesssim t^{\frac34-\kappa/2}$.
It follows that the second and third terms in \eqref{eq:RB1} are $O_{L^\infty}(t^{\frac14-\kappa/2})$.
For the final term,
recall $\kappa_1=\frac12(\omega+\frac12)+\frac32\kappa$.
Since the reconstruction of
$O_{\cD^{\frac32+2\kappa,\omega}_{0}}(t^{\ell})$
is $O_{L^\infty}(t^{\ell+\omega/2})$,
 by Proposition~\ref{prop:SPDE-Euler}  we obtain
\begin{equ}
|r_1|_{\cD^{\frac32+2\kappa,\omega}_0} \lesssim t^{\frac12-\kappa_1}\;,
\quad \text{hence} \quad
|(\CR r_1)(t)|_{L^\infty} \lesssim t^{\frac12-\kappa_1+\omega/2} = t^{\frac14 - 3\kappa/2}\;.
\end{equ}
Finally $P_t X(0) - X(0) = O_{L^\infty}(t)$,
which concludes the proof.
\end{proof}

\section{Short-time estimates for the YM heat flow}
\label{sec:Fs}

In this section we define a space $\SState$ parametrised by $(\alpha,\theta,\gamma,\delta)$
on which one can run the YM heat flow $\mcF_s $ introduced in Definition~\ref{def:CF}.
This space is related with the space denoted by $\state$ in \cite{CCHS_3D}, which we recall in Definition~\ref{def:state} and Remark~\ref{rem:state_embed},
but we impose here finer control on the leading order singularity of the flow.
We prove in Lemma~\ref{lem:Fs_remainder}
a small-$s$ approximation of $\mcF_s A$ by terms linear and quadratic in $A$ and 
an error $O_{L^\infty}(s^\nu)$ for suitable $\nu>0$.
We then show in Lemma~\ref{lem:Euler_F} that an $L^\infty$ perturbation $A\mapsto A+r$
changes $\mcF_s A$  by $P_s r$ 
plus an error under control.

Consider throughout this section $\gamma \in(0,1]$ and $\delta\in (0,1)$.

\subsection{State space}

\begin{definition}\label{def:quadnorm}
Define for $A\in\CD'(\T^3,\mfg^3)$
\begin{equ}\label{eq:CN_def}
\CN(A) \colon (0,\infty) \to \CC^\infty(\T^3, \mfg^3\otimes (\mfg^3)^3)\;,\qquad \CN_s(A) \eqdef P_s A\otimes \nabla P_s A\;.
\end{equ}
Define the set of line segments
\begin{equ}
\CL = \T^3 \times \{v\in\R^3\,:\,|v|\leq 1/4\}\;.
\end{equ}
For $\ell=(x,v)\in \T^3\times \R^3$, denote $|\ell|=|v|$.
For a normed space $F$, we denote by $\Omega_{\gr\gamma}$ the completion of $\CC(\T^3, F)$ under the norm\label{page:gr}
\begin{equ}
|f|_{\gr\gamma} = \sup_{\ell \in \CL} \frac{| \int_\ell f |}{|\ell|^\gamma}\;,
\end{equ}
where, for any $\ell=(x,v) \in \T^3\times\R^3$, we denote
\begin{equ}
\int_\ell f = \int_0^1 |v| f(x+tv)\mrd t \in F\;.
\end{equ}
The dependence of $\Omega_{\gr\gamma}$ on $F$ will always be clear from the context.\footnote{In contrast to \cite{CCHS_2D,CCHS_3D}, if $f\in \CC(\T^3,\mfg^3)$, then $\int_\ell f$ is $\mfg^3$-valued, so we do \emph{not} interpret $\int_\ell f$ as the usual line integral of a $1$-form (in which case $\int_\ell f$ would be $\mfg$-valued).}

Define further for $\delta>0$ and $\gamma\in(0,1]$\label{page:quadnorm}
\begin{equ}
\quadnorm{A;B}_{\gamma,\delta} \eqdef \sup_{s\in(0,1)} s^\delta |\CN_s A - \CN_s B|_{\gr\gamma}\;,\quad
\quadnorm{A}_{\gamma,\delta} = \quadnorm{A;0}_{\gamma,\delta}\;. 
\end{equ}
\end{definition}

\begin{remark}\label{rem:quadnorm-stronger}
Recall the definition from~\cite[Sec.~2.1]{CCHS_3D}
\begin{equ}
\fancynorm{A;B}_{\beta,\delta} \eqdef \sup_{s\in(0,1)} s^\delta|\CN_sA - \CN_sB|_{\CC^\beta}\;,
\end{equ}
whereas here in Definition~\ref{def:quadnorm} we replaced $|\cdot|_{\CC^\beta}$ by $|\cdot|_{\gr\gamma}$.
By \cite[Prop.~3.21]{Chevyrev19YM}, $|\cdot|_{\CC^{\gamma-1}}\lesssim |\cdot|_{\gr\gamma}$ and therefore
$\fancynorm{\cdot\,;\cdot}_{\gamma-1,\delta}\lesssim \quadnorm{\cdot\,;\cdot}_{\gamma,\delta}$.
\end{remark}
Consider furthermore $\alpha \in (0,1]$ and $\theta>0$. 
Recall the norm from~\cite[Sec.~2.3]{CCHS_3D}
\begin{equ}
\heatgr{A}_{\alpha,\theta} \eqdef \sup_{s\in(0,1)} |P_s A|_{\gr\alpha;<s^\theta}\;.
\end{equ}
where
\begin{equ}
|A|_{\gr\alpha;<r} \eqdef \sup_{\ell\in\mcL,\,|\ell|<r} \frac{|\int_\ell A|}{|\ell|^\alpha}\;.
\end{equ}
Recall also from \cite[Lem.~2.25]{CCHS_3D} that, with $\eta = (1+2\theta)(\alpha-1)$,
\begin{equ}[eq:eta_heatgr]
|A|_{\CC^\eta}\lesssim \heatgr{A}_{\alpha,\theta}\;.
\end{equ}

\begin{definition}\label{def:SState}
For $\alpha,\gamma\in (0,1]$, $\delta,\theta>0$, we let $\SState$ denote the set of all $A\in\CD'(\T^3,\mfg^3)$ such that
\begin{equ}
\SNorm(A) \eqdef \heatgr{A}_{\alpha,\theta}+\quadnorm{A}_{\gamma,\delta}<\infty\;.
\end{equ}
We equip $\SState$ with the metric
\begin{equ}
\SNorm(A,B) \eqdef \heatgr{A-B}_{\alpha,\theta}+\quadnorm{A;B}_{\gamma,\delta}<\infty\;.
\end{equ}
\end{definition}

\begin{remark}\label{rem:state_embed}
Following Remark \ref{rem:quadnorm-stronger},
the space $\state$ from \cite[Sec.~2.3]{CCHS_3D}
was defined similarly to $\SState$ with the main difference being that the metric on $\state$ is $\heatgr{A-B}_{\alpha,\theta}+\fancynorm{A;B}_{\beta,
\delta}$ for suitable $\beta<0$. 
\end{remark}

\begin{definition}
We say that $\alpha,\theta,\gamma,\delta$ satisfy \eqref{eq:init} if
\begin{equation}\label{eq:init}
\begin{split}
&\alpha\in(0,1/2)\;,\quad
\theta>0\;,\quad
\gamma \in (1/2,1]\;,\quad
\delta\in(0,1)\;,
\\
&\eta\eqdef (1+2\theta)(\alpha-1) > -2/3\;,
\\
&\mu\eqdef \gamma-1+2(1-\delta) \in (-1/2,0)\;,\quad
\text{and}
\quad
\eta+\mu>-1\;.
\end{split}\tag{$\mcI$}
\end{equation}
\end{definition}

\begin{remark}
The condition $\alpha<\frac12<\gamma$ is not strictly necessary for the results of this section,
but it provides a few simplifications,
e.g. it implies $\eta < -\frac12 < \mu$,
so we can estimate $|P_s A + \CP_s\star\CN A|_\infty$ by $s^{\eta/2}$ instead of $s^{(\eta\wedge\mu)/2}$,
see \eqref{e:PN-bound}.
\end{remark}

We use the notation $\Poly(K)$ to denote a term of the form $\Poly(K) = C K^q$ for some $C,q>0$.
(Here $C$ can be both large or small but $q$ will always be large.)
We also write $x\ll y$ to denote that there exists a small constant $c>0$ such that $x\leq cy$.

The following lemma gives the well-posedness of $\CF_s$ on $\SState$.
\begin{lemma}
\label{lem:CF_est}
Suppose \eqref{eq:init} holds.
For any $K>0$ and $0<s \ll 1\wedge \Poly(K^{-1})$, $\CF_s$ extends to a Lipschitz function from $\{A \in\SState\,:\,\SNorm(A)\leq K\}$
to $\CC^1(\T^3,\mfg^3)$.
Furthermore,
\begin{equ}
|\CF_s A|_\infty \lesssim s^{\frac\eta2}\SNorm(A)\;,\quad |\partial \CF_s A|_{\infty} \lesssim s^{\frac\eta2-\frac12}\SNorm(A)\;,
\end{equ}
where the proportionality constants depend only $\alpha,\theta,\gamma,\delta$. 
\end{lemma}

\begin{proof}
This immediately follows from \cite[Prop.~2.9]{CCHS_3D}   (the notation $\hat\beta$ therein is replaced by $\mu$ here),
the bounds (see \eqref{e:def-Cbeta})
\begin{equ}[e:CPA]
|P_s A|_\infty \lesssim s^{\frac\eta2} |A|_{\CC^\eta}\;,\quad |\partial P_s A|_\infty \lesssim s^{\frac\eta2-\frac12} |A|_{\CC^\eta}
\end{equ}
the estimate \eqref{eq:eta_heatgr}, the assumption $\eta<\mu$, as well as
 Remark~\ref{rem:quadnorm-stronger}.
\end{proof}

Recall from \cite[Lem.~2.48]{CCHS_3D} that, for any $0<\alpha\leq\gamma\leq 1$ and $\theta\geq0$, 
\begin{equ}
 |P_s A|_{\gr\gamma}
\lesssim
s^{\lambda}
 \heatgr{A}_{\alpha,\theta}^{\zeta}|A|^{1-\zeta}_{\CC^\eta}\;,
\end{equ}
where $\zeta = \frac{\gamma-1}{\alpha-1}\in [0,1]$ and
\begin{equ}[eq:lambda]
\lambda \eqdef (1-\zeta)\eta/2-\theta(1-\alpha)\zeta < 0\;.
\end{equ}
Combining with~\eqref{eq:eta_heatgr}, it follows that
\begin{equ}[eq:Psa]
|P_s A|_{\gr\gamma} \lesssim s^{\lambda} \heatgr{A}_{\alpha,\theta}\;.
\end{equ}

\begin{remark} 
Since $\SState$ is stronger than $\state$ 
in Remark~\ref{rem:state_embed} from \cite[Sec.~2.3]{CCHS_3D},
it inherits the  
 gauge equivalence relation $\sim$ from $\state$,
and the regularised Wilson loops $W_\ell[\CF_s(\cdot)]$  are gauge invariant observables with respect to $\sim$.
\end{remark}

\subsection{Perturbation of the YM heat flow}

In the rest of this section, consider $\alpha,\theta,\gamma,\delta$ satisfying \eqref{eq:init} and
\begin{equ}[e:cond-nu]
0<\nu \leq \min\Big\{\frac\eta2 + \frac{\mu}2 + \frac12\;,\;
1+3\eta/2\;,\;
\mu + \frac12\Big\}
\;.
\end{equ}
We will refer to $\alpha, \theta,\gamma,\delta,\nu$ as the \emph{Greeks}.
Recall the shorthand
$\CP_s\star \CN A = \int_0^s P_{s-r}\CN_r A \mrd r$ from \eqref{eq:P_star_def}.
Then
\begin{equ}[eq:PCN_bound]
|\CP_s\star \CN A|_{\gr\gamma} 
\lesssim
\int_0^s 
|\CN_r A |_{\gr\gamma}  \mrd r
\lesssim
\int_0^s 
r^{-\delta} \quadnorm{A}_{\gamma,\delta}\mrd r
\lesssim 
s^{1-\delta} \quadnorm{A}_{\gamma,\delta}\;,
\end{equ}
where we used that $P_s$ is a contraction for $|\cdot|_{\gr\gamma}$.

\begin{lemma}\label{lem:Fs_remainder}
Consider any $A\in \SState$. Then for all $s\ll 1\wedge  \Poly(\SNorm(A)^{-1})$,
\begin{equ}
\CF_s A = P_s A + \CP_s\star \CN^c A + R_s A
\end{equ}
where $\CN^c_r A = P_r A \partial P_r A$, i.e. the contraction of the tensors in $\CN A$ from \eqref{eq:CN_def}
that yields the quadratic term $[A_j,2\partial_j A_i - \partial_i A_j]$ in the YM heat flow \eqref{eq:YMH_flow},
and where
\begin{equ}
|R_s A|_\infty \lesssim s^{\nu} (\SNorm(A) + \SNorm(A)^3)
\end{equ}
with the proportionality constant depending only on the Greeks.
\end{lemma}

See  \cite[Prop.~2.9]{CCHS_3D} for a similar but weaker statement.
\begin{proof}
Observe that $R\equiv RA$ satisfies
\begin{equs}[e:fixpt-R]
R_s
= 
\int_0^s P_{s-r}\Big\{
 &
 (P_r A + \CP_r\star \CN^c A +R_r)\partial (P_r A + \CP_r\star \CN^c A+ R_r)
 \\
 &+( \CF_r A)^3
  \Big\} \mrd r
  - \CP_s \star \CN^c A
\;.
\end{equs}
Note that $\CN^c A$ is a linear function of $\CN A$.
Therefore, by \eqref{e:CPA} and \eqref{eq:PCN_bound},
\begin{equs}[e:PN-bound]
| \CP_s\star \CN^c A|_{\infty}
&\lesssim
\int_0^s |P_{s-r}(\CN_r A)|_{\infty}\mrd r
\lesssim
 \int_0^s (s-r)^{\frac{\gamma-1}{2}} 
|\CN_r A|_{\CC^{\gamma-1}} \mrd s 
 \\
 &\lesssim
 \int_0^s (s-r)^{\frac{\gamma-1}{2}} 
|\CN_r A|_{\gr\gamma} \mrd s 
 \\
& \lesssim
 \int_0^s (s-r)^{\frac{\gamma-1}{2}} 
 r^{-\delta}\quadnorm{A}_{\gamma,\delta}  \mrd s
  \asymp
  s^{\mu/2} \quadnorm{A}_{\gamma,\delta}\;,
\end{equs}
and similarly
\begin{equ}
| \d \CP_s\star \CN^c A|_{\infty}
\lesssim
  s^{\mu/2-\frac12} \quadnorm{A}_{\gamma,\delta}\;.
\end{equ}
Moreover, by Lemma \ref{lem:CF_est}, $|(\CF_rA)^3|_\infty \lesssim r^{3\eta/2}\SNorm(A)^3$.
Therefore, since
\begin{equ}
\CP_s\star \CN^c A=\int_0^s P_{s-r} (P_r A \partial P_r A) \mrd r
\;,
\end{equ}
the terms on the right-hand side of \eqref{e:fixpt-R} not involving $R$ are bounded in $L^\infty$ by a multiple of $\SNorm(A)+\SNorm(A)^3$ times
\begin{equ}
\int_0^s \{r^{\frac\eta2 + \frac{\mu}2 -\frac12} + r^{\mu-\frac12} + r^{3\eta/2}\} \mrd r
\asymp 
s^{\frac\eta2 + \frac{\mu}2 + \frac12} + s^{\mu + \frac12} +s^{1+3\eta/2}\;,
\end{equ}
where we used that $\frac\eta2 + \frac{\mu}2 + \frac12>0$ and $\mu+\frac12>0$ and $1+3\eta /2>0$ due to condition \eqref{eq:init}.

Consider the Banach space $\CB \equiv \CB_T$ with norm
\begin{equ}
|R|_{\CB} \eqdef \sup_{s\in (0,T)} \{s^{-\nu}|R_s|_\infty + s^{-\nu+\frac12}|\d R_s|_\infty\}\;.
\end{equ}
By the upper bound~\eqref{e:cond-nu} on $\nu$, the terms on the right-hand side of \eqref{e:fixpt-R} not involving $R$ are in a ball of radius $\asymp \SNorm(A)+\SNorm(A)^3$ in $\CB$.
It readily follows that
\begin{equs}
s^{-\nu}|R|_\infty
&\lesssim 
\SNorm(A)+\SNorm(A)^3 + s^{\kappa'}|R|_{\CB}(\SNorm(A)+\SNorm(A)^2)
\\
&\qquad + s^{\kappa''} |R|_{\CB}^2(1+\SNorm(A)) + s^{1+2\nu}|R|_{\CB}^3
\end{equs}
where
\begin{equs}
\kappa' &= \min\Big\{\frac{\eta }{2} +\frac12\;,\;
1+\eta
\Big\}
= \frac{\eta }{2} +\frac12\;,
\\
\kappa'' &= \min\Big\{
\frac12+\nu\;,\;
1+ \frac{\eta }{2}+\nu
\Big\}
= \frac12+\nu\;.
\end{equs}
The same bound holds for $s^{-\nu+\frac12}|\partial R|_\infty$,
so in conclusion
\begin{equ}
|R|_{\CB} \lesssim \SNorm(A)+\SNorm(A)^3 + T^\kappa (|R|_{\CB} + |R|_{\CB}^3)(1+\SNorm(A)^2)
\end{equ}
for some $\kappa>0$.
It follows by a continuity argument (in $T$), that for all $T \ll 1\wedge \Poly(\SNorm(A)^{-1})$,
one has $|R|_{\CB} \lesssim \SNorm(A)+\SNorm(A)^3$.
\end{proof}

\begin{lemma}\label{lem:Euler_F}
Suppose further that $\nu\leq 1-\delta$. Then for all $s\ll 1\wedge\Poly(\SNorm(A)^{-1})$
\begin{equ}[eq:CFa_CPa]
\CF_s A = P_s A + O_{\Omega_{\gr\gamma}[\mfg,\mfg]}(s^{\nu}(\SNorm(A) + \SNorm(A)^3))\;, \qquad P_s A = O_{\Omega_{\gr\gamma}}(s^{\lambda})\;.
\end{equ}
Suppose further that $\tilde A = A + r$ with $|r|_\infty = O(1)$.
Then for all $s \ll 1\wedge \Poly(\SNorm(A)^{-1})$
\begin{equ}[eq:CFa_tilde_a]
\CF_s \tilde A = \CF_s A + P_s r + O_{L^\infty[\mfg,\mfg]}(s^{\eta/2+\frac12}|r|_\infty)\;.
\end{equ}
The proportionality constant in each statement depends only on the Greeks.
\end{lemma}

\begin{proof}
The second claim of  \eqref{eq:CFa_CPa} follows from \eqref{eq:Psa}.
The first claim of  \eqref{eq:CFa_CPa} follows from
\eqref{eq:PCN_bound} (and the assumption  $\nu\leq 1-\delta$), Lemma \ref{lem:Fs_remainder}
and $|\cdot |_{\gr\gamma}\le | \cdot |_{\infty}$ for $\gamma\le 1$.

For \eqref{eq:CFa_tilde_a}, writing $\CF_s \tilde A = \CF_s A + Q_s$, we have 
\begin{equ}
\CF_s A + Q_s = P_s (A+r) + \int_0^s P_{s-u} \{ (\CF_u A + Q_u)\d(\CF_u A + Q_u)+(\CF_u A + Q_u)^3 \} \mrd u
\end{equ}
which implies that
\begin{equs}[eq:Q_int_form]
Q_s
&=
P_s r + \int_0^s P_{s-u}\Big(Q_u\d \CF_u A + (\CF_u A) \d Q_u + Q_u\d Q_u
\\
&\qquad \qquad+ (\CF_u A)^2 Q_u + (\CF_u A) Q_u^2 + Q_u^3 \Big)\mrd u\;.
\end{equs}
Recalling that, by Lemma \ref{lem:CF_est}, $|\CF_u A|_\infty \lesssim u^{\eta/2}\SNorm(A)$
and $|\CF_u A|_{\CC^1} \lesssim u^{\eta/2-\frac12}\SNorm(A)$
for $u\leq 1\wedge \Poly(\SNorm(A)^{-1})$, the integral in \eqref{eq:Q_int_form} is bounded in $L^\infty$ by a multiple of
\begin{equs}[eq:Q_int_bound]{}
&\int_0^s \{ |Q_u|_\infty u^{\eta/2-\frac12} + u^{\eta/2} |Q_u|_{\CC^1} + |Q_u|_\infty|Q_u|_{\CC^1}
\\
&\qquad + u^{\eta}|Q_u|_\infty + u^{\eta/2}|Q_u|^2_\infty + |Q_u|^3_\infty \}\mrd u
\\
&\lesssim
|Q|_\infty s^{\eta/2+\frac12} + s^{\eta/2+\frac12}|Q|_{L^\infty_{1/2} \CC^1} + s^{1/2}|Q|_\infty
|Q|_{L^\infty_{1/2} \CC^1}
\\
&\qquad
+s^{\eta+1}|Q|_\infty + s^{\eta/2+1}|Q|_\infty^2 + s |Q|_\infty^3\;,
\end{equs}
where we denote $|Q|_{L^\infty_{1/2} \CC^1} = \sup_{u\in (0,s)} u^{1/2}|Q_u|_{\CC^1}$
and the proportionality constant is polynomial in $\SNorm(A)$.
Furthermore, the integral in \eqref{eq:Q_int_form} is bounded in $\CC^1$ by the same quantity times $s^{-1/2}$.
Since each exponent of $s$ in \eqref{eq:Q_int_bound} is positive, it follows from a continuity argument (in $s$) that, for $s\ll 1\wedge \Poly(\SNorm(A)^{-1})$,
\begin{equ}
|Q|_\infty + |Q|_{L^\infty_{1/2} \CC^1} \lesssim |r|_\infty
\end{equ}
where the proportionality constant depends only on the Greeks.
Therefore, since the smallest exponent of $s$ on the right-hand side of \eqref{eq:Q_int_bound} is $\eta/2+\frac12$ and since we assumed $|r|_\infty = O(1)$,
\begin{equ}
Q_s = P_s r + O_{L^\infty}(s^{\eta/2+\frac12}|r|_\infty ) \;.
\end{equ}
It is furthermore clear that the final term takes values in $[\mfg,\mfg]$.
\end{proof}

\section{Regularised Wilson loops}
\label{sec:Wilson}

We fix in this section $\omega,\kappa$ as in Section \ref{sec:SPDE}
and also $\alpha,\theta,\gamma,\delta$ satisfying \eqref{eq:init} which also determine  $\eta,\mu$.
We furthermore fix a model and input distributions as in Section \ref{sec:SPDE}.

Suppose $\tilde \CX$ is given by \eqref{eq:X-formal} with initial condition $\tilde X(0)=( A(0),\Phi(0))$ and $h(0)$ is smooth.
Likewise let $\CX$ be given by \eqref{eq:X-formal} with the same initial condition $X(0) = \tilde X(0)$ but with $h(0)=0$.
We follow Notation \ref{not:t},
in particular we let $t \in (0,\tau)$ where $\tau^{-1}$ and all proportionality constants are polynomial in the quantity \eqref{eq:size}.

\begin{lemma}\label{lem:A_tilde_A} 
\begin{equs}
\tilde \CX
= B + c \bar h + O_{\cD^{\frac32+2\kappa,\omega}_{0}}(t^{3/2-\kappa_5})
\end{equs}
where $\bar h (t) = t P_t h(0)$ and 
where $B$ does not depend on $h(0)$
and $\kappa_5= \frac12 (\omega+\frac12) + \frac72 \kappa$ as in Proposition \ref{prop:SPDE-Euler}.
Likewise for $\CX$ except we take $h(0)=0$.
\end{lemma} 

\begin{proof}
We apply Proposition \ref{prop:SPDE-Euler} with $n=5$ to obtain
\begin{equ}
\tilde \mcX= B
 +c\,\CP Ph(0)
 + O_{\cD^{\frac32+2\kappa,\omega}_{0}}(t^{3/2-\kappa_5})\;,
\end{equ}
where $B=B_0 + \cdots + B_5$ and $(\CP Ph(0))(t) = \int_0^t P_{t-r} P_r h(0) \mrd r = t P_t h(0)$.
\end{proof}

\begin{remark}
In contrast to \cite[Prop. 8.8]{chevyrev2023invariant}, Lemma \ref{lem:A_tilde_A} does not have a non-explicit linear term (denote by $L_t h(0)$ in \cite{chevyrev2023invariant}).
This is because we use here $Ph(0)$ vs. $h(0)$ and
are only after an expansion to order $t^{3/2-\kappa_5}$ vs. $t^{2-\kappa}$ in \cite{chevyrev2023invariant}.
\end{remark}

Let $X = (A,\Phi)$ be the reconstruction of $\CX$ at time $t>0$,
and likewise for $\tilde X = (\tilde A,\tilde \Phi)$.
By Lemma~\ref{lem:A_tilde_A},
since the reconstruction of 
$O_{\cD^{\frac32+,\omega}_{0}}(t^{\ell})$
is $O_{L^\infty}(t^{\ell+\omega/2})$
and since $\frac32-\kappa_5 + \frac{\omega}2 = \frac54 - \frac72 \kappa$, one has
\begin{equ}
\tilde A = A + c t P_t h(0) + O_{L^\infty}(t^{5/4-7\kappa/2})\;.
\end{equ}

By~\eqref{eq:CFa_tilde_a} of Lemma~\ref{lem:Euler_F}, which is applicable because the final two terms are $O_{L^\infty}(1)$ for all $t \in (0,\tau)$ and $\tau$ as in Notation \ref{not:t},
\begin{equ}[e:CF-diff]
\CF_s \tilde A = \CF_s A
+ P_s \big(c t P_t h(0) + O_{L^\infty}(t^{5/4-7\kappa/2}) \big)
+ O_{L^\infty[\mfg,\mfg]}(s^{\eta/2+\frac12}t)\;.
\end{equ}

Recalling the Wilson loop $W_\ell = \Trace\hol (\cdot,\ell)$ from Definition~\ref{def:Wilson-loop},
our next goal is to compare the regularised Wilson loops $W_\ell (\CF_s \tilde A)$ and $W_\ell(\CF_s A)$.
Recall from  Definition~\ref{def:Wilson-loop} that the holonomy is (the endpoint of) a $G$-valued curve given as the solution of a linear ODE driven by a $\mfg$-valued curve.
We first analyse the change of the holonomy under a generic perturbation of a $\mfg$-valued curve.

For $p\geq 1$ and a normed space $F$, let $\CC^{\var p}([0,1],F)$ denote the space of continuous paths $f\in\CC([0,1],F)$ with finite $p$-variation
\begin{equ}
|f|_{\var p} \eqdef \sup_{P\subset [0,1]} \Big(\sum_{[s,t]\in P} |f(t) - f(s)|^p\Big)^{1/p}
\end{equ}
where the $\sup$ is over all partitions $P$ of $[0,1]$ into disjoint (modulo endpoints) interval.
For $p\in[1,2)$ and $\bgamma\in\CC^{\var p}([0,1],\mfg)$,
let $J^{\bgamma} \in \CC^{\var p} ([0,1],G)$ denote the solution to the linear ODE
\begin{equ}\label{eq:V_ODE}
\mrd J^{\bgamma}(x) = J^{\bgamma}(x)\mrd \bgamma(x)\;,\quad J^{\bgamma}(0)=\id\;,
\end{equ}
which is well-posed as a Young ODE~\cite{Lyons94}.
One has the following perturbation estimate.

\begin{lemma}\label{lem:path_perturb}
For $\bgamma,\bzeta\in \CC^{\var p}([0,1],\mfg)$ and $L\geq 4$
we have
\begin{equs}
{}&
J^{\bgamma+\bzeta}(1) = J^{\bgamma}(1) + \int_0^1 \mrd \bzeta(x) 
+
\int_0^1\int_0^x \{\mrd\bzeta(x)\mrd\bgamma(y)+ \mrd\bgamma(x)\mrd\bzeta(y)\}
\\
&\quad
+ O\{v(w^2+w^{L-1}) + w^L+w^{L+1} + v^{L+1} + v^2 (1+w + v + w^{L-3})\}\;.
\end{equs}
Here $v\eqdef |\bzeta|_{\var p}$, $w\eqdef |\bgamma|_{\var p}$, 
and the proportionality constants depend only on $p$ and $L$.
\end{lemma}

\begin{proof}
Denoting $I^{\bgamma}=\sum_{k=1}^{L-1} \int_0^1\ldots\int_0^{x_{k-1}} \mrd \bgamma(x_{k})\ldots \mrd \bgamma(x_1)$,
by linearity of the ODE~\eqref{eq:V_ODE}, for every $L\geq 1$ 
\begin{equ}
J^{\bgamma}(1) = \id
+ I^{\bgamma} +
\int_0^1\cdots\int_0^{x_{L-1}} J^{\bgamma}(x_{L})\mrd \bgamma(x_{L}) \mrd \bgamma(x_{L-1})\ldots \mrd \bgamma(x_1)\;.
\end{equ}
The proof is then similar with \cite[Lem.~8.15]{chevyrev2023invariant},
in particular since $J^{\bgamma}$ takes values in a compact set the last term above is bounded by $O(w^L + w^{L+1})$.
The only difference is that
when we compare $I^{\bgamma+\bzeta}$ with  $I^{\bgamma}$,
instead of keeping all the terms linear in $\bzeta$ (as denoted by
$P^{\bgamma}(\bzeta)$ therein), we now write 
\begin{equs}
I^{\bgamma+\bzeta} &= I^{\bgamma} + \int_0^1\mrd \bzeta(x) 
+\int_0^1\int_0^x \{\mrd\bzeta(x)\mrd\bgamma(y)+ \mrd\bgamma(x)\mrd\bzeta(y)\}
\\
&\qquad 
+O\{v^2+v(w^2+w^{L-1}) \}
+ O\{v^2 (w + v + w^{L-3} + v^{L-3})\}\;.
\end{equs}
Here $O\{v^2+v(w^2+w^{L-1}) \}$
is a term arising from $\int_0^1\int_0^x \mrd \bzeta(y)\mrd \bzeta(x) $ and
 the 
integrals with one instance of $\bzeta$ and $n$ instances of $\bgamma$ with
$2 \le n \le L-1$.
\end{proof}

Consider the loop $\ell\colon [0,1]\to\T^3$,  $\ell(x)=(x,0,0)$.
Recall that we fixed $\alpha,\theta,\gamma,\delta,\eta,\mu$ satisfying \eqref{eq:init}.
Recall $\lambda<0$ from \eqref{eq:lambda}. Let $\nu \in (0,1-\delta)$ satisfy \eqref{e:cond-nu}.
Recall that Notation \ref{not:t} is in place.

For a $\mfg$-valued $1$-form $a$, recall the line integral $\ell_a\colon [0,1]\to\mfg$ defined by \eqref{eq:ell_A}.
Denoting by $|\cdot|_{\Hol\gamma}$ is the usual H\"older norm,
it is obvious that
\begin{equ}[eq:ell_var_p]
|\ell_a|_{\var{\frac1\gamma}} \leq |\ell_a|_{\Hol{\gamma}} \leq |a|_{\gr{\gamma}}
\;.
\end{equ}
The following is the main result of this section.

\begin{lemma}\label{lem:Wilson_expansion}
For all $s\ll 1\wedge \Poly(\SNorm(A)^{-1})$ and $L \geq 4$,
  \begin{equs}
{}& W_\ell (\CF_s \tilde A)
= W_\ell(\CF_s A) + t\Trace \int_\ell c h(0) 
+ t \Trace \int_{[0,1]^2} \mrd \ell_{A(0)}(x_1) \mrd \ell_{c h(0)}(x_2)
\label{eq:reg_W_ell}
\\
& + O\Big(
t^{5/4-7\kappa/2} + ts + ts^\lambda \heatgr{\Psi_t^\ym}_{\alpha,\theta} + ts^\lambda \heatgr{\CP_t \star (\Psi \d \Psi)^\ym}_{\alpha,\theta} + |A(0)|_{L^\infty} s^{\eta/2+\frac12} t
\\
&\quad + s^\nu t(\SNorm(A)+\SNorm(A)^3) + t(u^2 + u^{L-1}) + s^{\nu L} + u^L + u^{L+1} + t^2 u \Big)\;,
\end{equs}
where $u= s^\lambda \heatgr{A}_{\alpha,\theta}$,
$\Psi = \CR\bPsi$ as in \eqref{eq:Psi},
and we recall that $Y^\ym$ is the $\mfg^3$-component of $Y\in \mfg^3\oplus\higgsvec$.
\end{lemma}

We remark that although $\lambda<0$, in Section~\ref{sec:expectation-loop} we will
choose $A(0)$ (and thus $\heatgr{A}_{\alpha,\theta}$) sufficiently small, so that $u$ will be small.

\begin{proof}
Define the $\mfg$-valued curve
\[
\bgamma(x) = \int_0^x (\CF_s A)_1(y,0,0)\mrd y = \ell_{\CF_s A}(x)\;.
\]
By \eqref{e:CF-diff}, $ \ell_{\CF_s \tilde A}(x) = \bgamma(x)+\bzeta(x)$ with
 $\bzeta = \ell_{D_h+D_{\err}}$ where
\begin{equs}[eq:Ds]
D_h & \eqdef ctP_{t+s}h(0) = cth(0) + O_{L^\infty}(t(t+s)) \;, 
\\
D_{\err}
&\eqdef O_{L^\infty[\mfg,\mfg]}(s^{\eta/2+\frac12}t) + O_{L^\infty}(t^{5/4-7\kappa/2}) \;.
\end{equs}
By definition one has
\[
W_\ell (\CF_s A) =  \Trace J^{\bgamma}(1)\;,
\qquad
W_\ell (\CF_s \tilde A) = \Trace J^{\bgamma+\bzeta}(1)\;.
\]
By Lemma~\ref{lem:path_perturb}, and taking trace, for $p=1/\gamma \in [1,2)$ (where $\gamma$ is the exponent in Section~\ref{sec:Fs}, in particular \eqref{eq:init})
and any $L\geq 4$,
\begin{equs}[eq:error_W]
{}& W_\ell (\CF_s \tilde A)
= W_\ell(\CF_s A) 
+ \Trace  \Big(\int_0^1 \mrd\bzeta(x) \Big)
\\
&\quad +\Trace \int_0^1\int_0^x \{\mrd\bzeta(x)\mrd\bgamma(y)+ \mrd\bgamma(x)\mrd\bzeta(y)\}
\\
&\quad
+O\{v(w^2+w^{L-1}) + w^L+w^{L+1} + v^{L+1} + v^2 (1+w + v + w^{L-3})\}
\end{equs}
where $v= |\bzeta|_{\var{\frac1\gamma}}$ and $w= |\bgamma|_{\var{\frac1\gamma}}$.
We now match the right-hand side of \eqref{eq:error_W} with that of the lemma statement.

First, using that $\Trace [\mfg,\mfg] = \{0\}$, we have
\begin{equ}
\Trace  \Big(\int_0^1 \mrd\bzeta(x) \Big)
= t\Trace\int_\ell c h(0) + O(t^{5/4-7\kappa/2}+ts)
\;,
\end{equ}
which gives the second term and the $O(t^{5/4-7\kappa/2}) + O(ts)$ terms on the right-hand side of \eqref{eq:reg_W_ell}.

Next, by the cyclic property of trace
\begin{equ}
\Trace \int_0^1\int_0^x \{\mrd\bzeta(x)\mrd\bgamma(y)+ \mrd\bgamma(x)\mrd\bzeta(y)\}
= 
\Trace \int_{[0,1]^2} \{\mrd\bzeta(x)\mrd\bgamma(y)\}\;.
\end{equ}
Consider the right-hand side. By Lemma \ref{lem:Euler_F},
\begin{equ}
\CF_s A = P_s A + O_{\gr\gamma}(s^\nu(\SNorm(A)+\SNorm(A)^3))
\end{equ}
and, by Lemma \ref{lem:a_remainder},
\begin{equ}[eq:a_expansion]
A = A(0) + \Psi^\ym_t + \CP_t \star (\Psi \d \Psi)^\ym + O_{L^\infty}(t^{1/4-3\kappa/2}) \;.
\end{equ}
Therefore, since $P_s A(0) = A(0) + O_{L^\infty}(s)$ and $P_s$ is a contraction on $\Omega_{\gr\gamma}$
and $|P_s\Psi^\ym_t|_{\gr\gamma} \lesssim s^\lambda \heatgr{\Psi^\ym_t}_{\alpha,\theta}$ by~\eqref{eq:Psa} where $\lambda<0$ is as in \eqref{eq:lambda}, and likewise for $\CP_t \star (\Psi \d \Psi)^\ym$,
we have
\begin{equs}
\int_0^1 \mrd\bgamma(x) &= \int_0^1\mrd \ell_{A(0)}(x) +
 O\left(s^\lambda\heatgr{\Psi^\ym_t}_{\alpha,\theta} + s^\lambda\heatgr{\CP_t \star (\Psi \d \Psi)^\ym}_{\alpha,\theta}\right)
\\
&\quad +
O\left(t^{1/4-3\kappa/2} + s^\nu(\SNorm(A)+\SNorm(A)^3)\right)\;.
\end{equs}
Furthermore
\begin{equ}
\int_0^1 \mrd\bzeta(y) = t\int_0^1\mrd \ell_{c h(0)}(y)
+
O(t^{5/4-7\kappa/2} + s^{\eta/2+\frac12}t)\;,
\end{equ}
where we used \eqref{eq:Ds} and the fact that $t(t+s) \leq t^{5/4-7\kappa/2} + s^{\eta/2+\frac12}t$.
Therefore
\begin{equs}
\Trace \int_{[0,1]^2} \{\mrd\bzeta(x)
&
\mrd\bgamma(y)\}
= 
t \Trace \Big( \int_{[0,1]^2} \mrd \ell_{A(0)}(x_1) \mrd \ell_{c h(0)}(x_2) \Big)
\\
& + O(ts^\lambda \heatgr{\Psi^\ym_t}_{\alpha,\theta} + ts^\lambda \heatgr{\CP_t \star (\Psi \d \Psi)^\ym}_{\alpha,\theta} + t^{5/4-7\kappa/2})
\\
& + O(ts^\nu(\SNorm(A)+\SNorm(A)^3) + s^{\eta/2+\frac12}t|A(0)|_{L^\infty})\;,
\end{equs}
which gives the 3rd term on the right-hand side of the lemma and the corresponding error terms.

It remains to analyse the error term in \eqref{eq:error_W}.
Note that, by~\eqref{eq:Psa},~\eqref{eq:CFa_CPa}, and~\eqref{eq:ell_var_p},
\begin{equ}[eq:w_bound]
w = |\bgamma|_{\var{\frac1\gamma}}
 \leq |\CF_s A|_{\gr\gamma}
  = |\CP_s A|_{\gr\gamma} + O(s^\nu)
 = O(s^{\lambda}) \heatgr{A}_{\alpha,\theta} + O(s^\nu) \;.
\end{equ}
Furthermore, recalling that $\eta>-\frac23$,
\begin{equ}[eq:v_bound]
v = |\bzeta|_{\var{\frac1\gamma}} \leq |D_h+D_{\err}|_{\gr\gamma} \lesssim s^{\eta/2+\frac12}t + t
= O(t)\;.
\end{equ}
By substituting $u=s^{\lambda} \heatgr{A}_{\alpha,\theta}$ and \eqref{eq:w_bound}-\eqref{eq:v_bound} into \eqref{eq:error_W} and dropping irrelevant terms (in particular remarking that $t^2 u^{L-3}\lesssim u^L + t^2 u$),
we can absorb the $O(\cdots)$ term of \eqref{eq:error_W} into that of \eqref{eq:reg_W_ell}.
\end{proof}

\begin{remark}
Connecting to the discussion in Section \ref{sec:idea},
the second and third term on the right-hand of \eqref{eq:reg_W_ell} are our `good terms' that we will use to exhibit a difference between $W_\ell (\CF_s \tilde A)$ and $W_\ell(\CF_s A)$ after taking expectations with suitable choices for $A(0)$ and $h(0)$.
All the terms in $O(\cdots)$ are the `remainders' that we will show are of lower order.
\end{remark}

\section{Estimates on quadratic terms}
\label{sec:quadratic}

In this section we estimate $\quadnorm{A}_{\gamma,\delta}$ when $A$ is of the form $A=\Psi_t+R_t$ where $\Psi$ solves the stochastic heat equation (SHE) with zero initial condition and $R$ is a remainder which is small in $\CC^{\bar\eta}$ for $\bar\eta<0$ close to $0$.

\subsection{Deterministic estimates}

\begin{lemma}\label{lem:quadnorm_pert}
Suppose $\gamma \in(0,1]$, $\delta > 0$, and $\eta\leq\bar\eta\leq 0$ with $\eta+\bar\eta>1 -2\delta$.
Then
\begin{equ}
\quadnorm{A+B}_{\gamma,\delta} \lesssim \quadnorm{A}_{\gamma,\delta} + |B|_{\CC^{\bar\eta}}(|A|_{\CC^{\eta}}+|B|_{\CC^{\bar\eta}})\;.
\end{equ}
\end{lemma}

\begin{proof}
We have
\begin{equ}
\CN_s (A+B) = \CN_s A + \CN_s B + P_s A \otimes\nabla P_s B + P_s B \otimes\nabla P_s A\;.
\end{equ}
By standard heat flow estimates, we estimate the cross terms in $L^\infty$ by
\begin{equ}
|P_s A \otimes\nabla P_s B|_\infty + |P_s B \otimes\nabla P_s A|_\infty
\lesssim s^{(\eta+\bar\eta-1)/2}|A|_{\CC^\eta} |B|_{\CC^{\bar\eta}}\lesssim s^{-\delta}|A|_{\CC^\eta} |B|_{\CC^{\bar\eta}}\;,
\end{equ}
where we used $\eta+\bar\eta>1 -2\delta$.
Moreover, $|\CN_s B|_\infty \lesssim s^{(2\bar\eta-1)/2} |B|_{\CC^{\bar\eta}}^2 \lesssim s^{-\delta} |B|_{\CC^{\bar\eta}}^2$
since $2\bar\eta>1 -2\delta$.
Since $|\cdot|_{\gr\gamma} \leq |\cdot|_\infty$, the conclusion follows.
\end{proof}

\subsection{Stochastic estimates}

Consider now $\Psi$ solving the SHE, i.e. $\d_t \Psi = \Delta \Psi + \xi$, where $\xi$ is a white noise, with $\Psi_0=0$.
For $s\in (0,1), t\in [0,1]$, let us denote
\begin{equ}
Z_{s,t} = s^\delta \CN_s \Psi_t\;.
\end{equ}

\begin{theorem}\label{thm:Z_Holder}
Let $\delta \in (\frac34,1)$ and $\kappa\in(0,\frac12)$ such that $4\delta-3-2\kappa>0$.
Then for all $0<\gamma<2\delta-1-\kappa$, $0<\bar\kappa<\kappa$,
and $p\in[1,\infty)$
\begin{equ}
\E \Big[\Big|\sup_{(s,t)\neq(\bar s,\bar t)} \frac{|Z_{s,t}-Z_{\bar s,\bar t}|_{\gr\gamma}}{(|t-\bar t| + |s-\bar s|)^{\bar\kappa}}\Big|^p\Big]^{1/p} < \infty
\end{equ}
\end{theorem}

\begin{proof}
Combining Lemmas \ref{lem:Z_diff_moment} and~\ref{lem:Kolmogorov_A} below, we obtain
\begin{equ}
( \E |Z_{s,t}-Z_{\bar s,\bar t}|^p_{\gr\gamma} )^{1/p} \lesssim
(|t-\bar t|+|s-\bar s|)^{\kappa/2}\;.
\end{equ}
The conclusion follows from Kolmogorov's criterion~\cite[Thm.~4.23]{Kallenberg21} applied to the stochastic process $Z \colon (0,1)\times [0,1] \to \Omega_{\gr\gamma}$.
\end{proof}

\begin{corollary}\label{cor:SHE_quadnorm}
Consider $\delta\in (\frac34,1)$ and $0<\gamma<2\delta-1$.
Then there exists $\bar\kappa>0$ such that, for all $p\in[1,\infty)$,
\begin{equ}
\E
\Big|
\sup_{t\in[0,1]}t^{-\bar\kappa}\quadnorm{\Psi_t}_{\gamma,\delta}
\Big|^p < \infty\;.
\end{equ}
\end{corollary}

\begin{proof}
There exist $\kappa,\bar\kappa$ satisfying the conditions of Theorem~\ref{thm:Z_Holder}.
Remark that
$\sup_{t\in[0,1]}t^{-\bar\kappa}\quadnorm{\Psi_t}_{\gamma,\delta} = \sup_{t\in[0,1]} \sup_{s\in(0,1)} t^{-\bar\kappa} |Z_{s,t}|_{\gr\gamma}$.
Furthermore $Z_{s,0}=0$ for all $s\in (0,1)$ since $\Psi(0)=0$.
The conclusion thus follows by applying Theorem~\ref{thm:Z_Holder} with $s=\bar s$ and $\bar t=0$.
\end{proof}

\begin{lemma}
Let $C_{s,\bar s;t,\bar t}(x) \eqdef \E 
\big( Z_{s,t} - Z_{\bar s, \bar t}
 \big)(0) \otimes 
\big( Z_{s,t} - Z_{\bar s,\bar t}
\big)(x)$.
For $\delta\in (\frac12,1)$ and $\kappa \in [0,\frac12)$,
we have uniformly in $s,\bar s \in (0,1)$, $t,\bar t\in [0,1]$, and $0\neq x\in\T^3$,
\begin{equ}[eq:C_bound]
|C_{s,\bar s;t,\bar t}(x) |
\lesssim (|s-\bar s|+|t-\bar t|)^{\kappa}  |x|^{4\delta-4-2\kappa}
\end{equ}
and
\begin{equ}[eq:C_diff_bound]
|\nabla C_{s,\bar s;t,\bar t}| 
\lesssim
(|s-\bar s|+|t-\bar t|)^{\kappa}  |x-y|^{4\delta-5-2\kappa}
\;.
\end{equ}
\end{lemma}

\begin{proof}
  The first bound is due to the proof of \cite[Lem.~3.4]{CCHS_3D} (see, in particular, the bound after Eq.~(3.6) therein).
  For the second bound, as in \cite[below Eq.~(3.5)]{CCHS_3D} one has, for $\kappa\in [0,\frac12)$,
  \begin{equ}
  |\nabla C_{r,s}(x)| \lesssim   |x|^{-2}\;,
  \qquad
|\nabla(C_{r,r}-C_{r,s})(x)| 
\lesssim 
|r-s|^{\kappa} |x|^{-2-2\kappa}
\end{equ}
uniformly in $r,s\in (0,1)$ and $x\in \T^3\setminus\{0\}$,
where 
$C_{r,s}(x) = \E \scal{ \Psi(r,0),\Psi(s,x)}_{\mfg^3}$.
Then  the bounds in \cite[(3.6)]{CCHS_3D}  again hold, with each kernel $C$ on the left-hand sides therein replaced by $\nabla C$ and each  $2\alpha-1$ in the exponent on the right-hand sides therein replaced by $2\alpha-2$.
Then the same argument as in \cite[Lem.~3.4]{CCHS_3D} yields \eqref{eq:C_diff_bound}.
\end{proof}

Recall the space of line segments $\CL = \T^3 \times \{v\in\R^3\,:\,|v|\leq \frac14\}$.
Define the metric $d$ on $\T^3 \times \R^3$ by
\begin{equ}
d((x,v),(\bar x,\bar v)) \eqdef |x-\bar x|\vee|x+v - (\bar x+\bar v)|\;.
\end{equ}
We say that $\ell,\bar\ell \in \CL$ are \emph{far} if $d(\ell,\bar\ell) > \frac14 (|\ell| \wedge |\bar\ell|)$.

\begin{lemma}\label{lem:Z_diff_moment}
Let $\delta,\kappa$ be as in Theorem \ref{thm:Z_Holder}.
Then for all $p \in [1,\infty)$,
uniformly over $s,\bar s \in (0,1)$, $t,\bar t\in [0,1]$, and $\ell\in\CL$
\begin{equ}[eq:Z_ell_bound]
\Big(\E \Big|\int_\ell (Z_{s,t}-Z_{\bar s,\bar t})\Big |^p\Big)^{1/p} \lesssim (|t-\bar t|+|s-\bar s|)^{\kappa}
|\ell|^{2\delta - 1 -\kappa}
\end{equ}
and over all $\ell,\bar\ell\in\CL$
\begin{equ}[eq:Z_P_bound]
\Big( \E \Big|\Big(\int_{\ell} - \int_{\bar\ell}\Big)(Z_{s,t}-Z_{\bar s,\bar t})\Big|^p \Big)^{1/p}
\lesssim
(|t-\bar t|+|s-\bar s|)^{\kappa/2}
d(\ell,\bar\ell)^{2\delta - 3/2 - \kappa}\;.
\end{equ}
\end{lemma}

\begin{proof}
By equivalence of moments in a fixed Wiener chaos, it suffices to consider $p=2$.
Let us fix $s,\bar s \in (0,1)$, $t,\bar t\in [0,1]$ and write $C(x)=C_{s,\bar s;t,\bar t}(x)$.
Integrating \eqref{eq:C_bound} bound against a line $\ell=(x,v)$, we obtain
\begin{equs}
\E \Big|\int_\ell (Z_{s,t}-Z_{\bar s,\bar t})\Big |^2
&=
|\ell|^2 \int_{[0,1]^2}  C((r-\bar r)v)
\mrd r \mrd \bar r
\\
&\lesssim
(|t-\bar t|+|s-\bar s|)^{\kappa}
|\ell|^{4\delta - 2 -2\kappa}
\int_{[0,1]^2} |r-\bar r|^{4\delta - 4 -2\kappa} \mrd r \mrd \bar r
\\
&\lesssim
(|t-\bar t|+|s-\bar s|)^{\kappa}
|\ell|^{4\delta - 2 -2\kappa}\;,
\end{equs}
where in the final bound we used $4\delta - 4 - 2\kappa > -1$.
This proves \eqref{eq:Z_ell_bound}.

We now prove \eqref{eq:Z_P_bound}.
Suppose first that $\ell,\bar\ell$ are far. Then the claim follows from \eqref{eq:Z_ell_bound} and the triangle inequality since $2\delta - 3/2 -\kappa < 2\delta - 1 -\kappa$ and $d(\ell,\bar\ell) \gtrsim |\ell|+|\bar\ell|$.

Consider now $\ell=(x,v),\bar\ell=(\bar x,\bar v)\in\CL$ not far.
Consider $\ell' = (x,v')$ where $v' = \bar x + \bar v - x$, see Figure \ref{fig:ells}.
(Note that $\ell'$ might not be in $\CL$ since it is possible that $|v'|>\frac14$.)
Then we can write $\int_{\ell}-\int_{\bar\ell} = \int_\ell - \int_{\ell'} + (\int_{\ell'}-\int_{\bar\ell})$ and remark that $d(\ell,\ell')\vee d(\ell',\bar\ell)=d(\ell,\bar\ell)$.
So it suffices to prove \eqref{eq:Z_P_bound} with $\bar\ell$ replaced by $\ell'$. 

Moreover, we can write $\int_{\ell'} f = \int_{\ell''} f + \int_r f$ where $\ell''=(x,v'')$ and $v''=cv'$ for $c\geq 0$ such that $|v''|=|v|$, and $r\in\CL$ is the `remainder' with $|r|\leq d(\ell,\ell')$,
see again Figure \ref{fig:ells}.
Then we write $\int_\ell-\int_{\bar\ell} = (\int_{\ell'}-\int_{\bar\ell}) + \int_r$
and note that $(\E|\int_r (Z_{s,t}-Z_{\bar s,\bar t})|^2)^{1/2}\lesssim (|t-\bar t|+|s-\bar s|)^{\kappa}
d(\ell,\ell')^{2\delta - 1 -\kappa}$ by \eqref{eq:Z_ell_bound},
which is smaller than the right-hand side of \eqref{eq:Z_P_bound} since
$d(\ell,\ell') \leq |\ell|$.
\begin{figure}[ht]
\centering
\begin{tikzpicture}[scale=3, every node/.style={font=\small}]

  \coordinate (x) at (0,0);
  \coordinate (xplusv) at (3,0); 

  \coordinate (xb) at (0.1,0.3);
  \coordinate (xbplusvb) at (3.5,0.5); 

  \coordinate (ellprime) at (3.5,0.5);
  
  \coordinate (elldprime) at (2.97,0.42);
  
  \draw[->, thick] (x) -- (xplusv) node[midway, below] {$\ell=(x,v)$};
  
  \draw[->, thick] (xb) -- (xbplusvb) node[midway, above] {$\bar\ell=(\bar x,\bar v)$};
  
  \draw[->, thin] (x) -- (ellprime) node[midway, below right] {$\ell'$};
  
  \draw[->, thick, blue] (x) -- (elldprime) node[midway, below left] {$\ell''$};
  
  \draw[->, red, thick] (elldprime) -- (ellprime) node[midway, below] {$r$};
  
  \node[below left] at (x) {$x$};
  \node[above left] at (xb) {$\bar x$};
  
  \node[below right] at (xplusv) {$x+v$};
  \node[above right] at (xbplusvb) {$\bar x+\bar v$};
  
\end{tikzpicture}
\caption{Example of $\ell,\bar\ell$ not far and corresponding $\ell',\ell''$ and $r$. By construction, $\ell'$ is the concatenation of $\ell''$ and $r$.}
\label{fig:ells}
\end{figure}
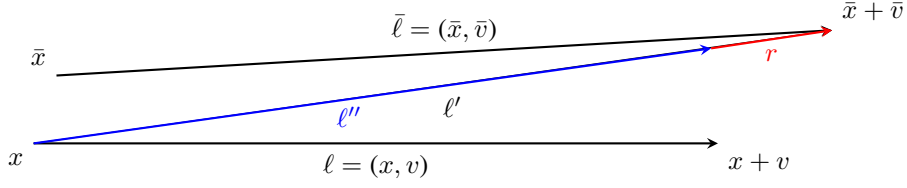

In conclusion, it suffices to prove \eqref{eq:Z_P_bound} for $\ell=(x,v),\bar\ell=(\bar x,\bar v)$ not far and with $x=\bar x$ and $|\ell|=|\bar\ell|$.\footnote{In \cite[Sec.~3.1]{CCHS_2D} such $\ell,\bar\ell$ were said to \emph{form a vee}.}
Let us denote $h \eqdef |v-\bar v|/|\ell| \leq \frac14$.
Then
\begin{equs}
\E \Big| \Big(\int_\ell - \int_{\bar\ell}\Big)(Z_{s,t}-Z_{\bar s,\bar t})\Big|^2
&= |\ell|^2\int_{[0,1]^2}
\{C((r-\bar r)v) - C(rv - \bar r \bar v)
\\
&\qquad\qquad - C(r\bar v - \bar r v) + C((r - \bar r)\bar v)\} \mrd r \mrd \bar r
\label{eq:C_diff_int}
\\
&=
|\ell|^2\int_{[0,1]^2}
\{2C((r-\bar r)v) - 2C(rv - \bar r \bar v)\} \mrd r \mrd \bar r\;,
\end{equs}
where we used that $r\bar v - \bar r v = rv - \bar r\bar v$ by symmetry.

Using the bound \eqref{eq:C_bound}, the integral over the region $|r-\bar r| \leq h$ in \eqref{eq:C_diff_int}
is bounded by a multiple of
\begin{equs}
{}&(|t-\bar t|  +|s-\bar s|)^{\kappa}
|\ell|^{4\delta -2 -2\kappa}
\int_0^{h} r^{4\delta - 4 -2\kappa} \mrd r 
\\
&\asymp
(|t-\bar t|+|s-\bar s|)^{\kappa}
|\ell|^{4\delta -2 -2\kappa}
h^{4\delta - 3 -2\kappa}
\\
&= (|t-\bar t|+|s-\bar s|)^{\kappa}
|\ell|
|v-\bar v|^{4\delta - 3 -2\kappa}
\;,
\end{equs}
where we again used $4\delta - 4 - 2\kappa > -1$.
On the other hand, using the bound \eqref{eq:C_diff_bound} and the fact that $|r-\bar r||v| \leq |r v - \bar r \bar v|$ for all $r,\bar r \in [0,1]$, the integral over the region $|r-\bar r| > h$ in \eqref{eq:C_diff_int}
is bounded by a multiple of
\begin{equs}
{}
&(|t-\bar t|+|s-\bar s|)^{\kappa}
|\ell|^{4\delta - 3 -2\kappa}|v-\bar v|
\int_{h}^1 r^{4\delta - 5 -2\kappa} \mrd r
\\
&\asymp
(|t-\bar t|+|s-\bar s|)^{\kappa}
|\ell|^{4\delta - 3 -2\kappa}|v-\bar v| h^{4\delta - 4 -2\kappa}
\\
&= (|t-\bar t|+|s-\bar s|)^{\kappa}
|\ell| |v-\bar v|^{4\delta - 3 -2\kappa}
\;.
\end{equs}
In conclusion, we obtain
\begin{equ}
\E \Big| \Big(\int_\ell - \int_{\bar\ell}\Big)(Z_{s,t}-Z_{\bar s,\bar t})\Big|^2 \lesssim
(|t-\bar t|+|s-\bar s|)^{\kappa}
|\ell| |v-\bar v|^{4\delta - 3 -2\kappa}\;.
\end{equ}
\end{proof}

The following lemma is a Kolmogorov-type criterion which is similar to \cite[Lem.~4.11]{CCHS_2D} and \cite[Lem.~3.12]{CCHS_3D} but somewhat simpler.

\begin{lemma}\label{lem:Kolmogorov_A}
Suppose $A$ is a $\CC(\T^3,F)$-valued random variable, where $F$ is a normed space.
Consider $\alpha,\beta \in (0,1]$. Suppose that, for every $p\in [1,\infty)$ there exists $M_p>0$ such that, for all $\ell,\bar\ell\in\CL$,
\begin{equ}[eq:A_ell]
\Big(\E\Big|\int_\ell A\Big|^p\Big)^{1/p}
\leq M_p |\ell|^{\alpha}\;,
\end{equ}
and
\begin{equ}[eq:A_P]
\Big(\E\Big|\Big(\int_\ell-\int_{\bar\ell}\Big)A\Big|^p \Big)^{1/p}
\leq M_p d(\ell,\bar\ell)^{\beta}\;.
\end{equ}
Then, for every $\gamma \in (0,\alpha)$ and $p\in [1,\infty)$, there exists $\lambda > 0$, depending only on $p,\alpha,\beta,\gamma$, such that
\begin{equ}
\big(\E|A|_{\gr{\gamma}}^p\big)^{1/p} \leq \lambda M_p\;.
\end{equ}
\end{lemma}
\begin{proof}
It suffices to consider $\beta\leq\alpha$.
For $N \geq 1$ let $D_{N}$ denote the set of line segments in $\CL$
whose start and end points have dyadic coordinates of scale $2^{-N}$,
and let $D=\cup_{N\geq 1} D_{N}$.
Consider $\omega \geq 1$.
For $\ell\in D$, let $k\leq 0$ be an integer such that $|\ell|\asymp 2^{k/\omega}$.
Then there exist $m\geq 0$ finite and $\ell_{k+i} \in D_{k+i}$ for $0\leq i \leq m$ such that
$\int_\ell A = \int_{\ell_k}A + \sum_{i=1}^m (\int_{\ell_{k+i}}-\int_{\ell_{k+i-1}})A$
and such that $|\ell_{k+i}|\asymp |\ell|$ and
$d(\ell_{k+i-1},\ell_{k+i}) \leq K2^{(-k-i)}$ for a constant $K>0$.
Using the elementary bound $\sum_{i\geq 0} a_i \lesssim \sup_{i\geq 0} 2^{\gamma i/\omega} a_i $ for $a_i\geq 0$ and $\gamma\in(0,\alpha)$,
it follows that
\begin{equ}\label{eq:A_dyadic_bound}
\sup_{\ell\in D} \frac{|\int_\ell A|}{|\ell|^{\gamma}}
\lesssim
\sup_{N\geq 1} \sup_{\substack{a\in D_{N}\\|a|\leq K2^{-N/\omega}}}
\frac{|\int_a A|}{2^{-\gamma N/\omega}}
+
\sup_{N\geq 1} \sup_{\substack{a,b \in D_{N}\\d(a,b)\leq K2^{-N}}}
\frac{|(\int_a - \int_b)A|}{2^{-\gamma N/\omega}}\;.
\end{equ}
Observe that $|D_{N}| \asymp 2^{6N}$.
Therefore,
raising both sides of~\eqref{eq:A_dyadic_bound} to the power $p$ and replacing the suprema on the right-hand side by sums, we obtain from \eqref{eq:A_ell}-\eqref{eq:A_P}
\begin{equ}
\E
\Big|\sup_{\ell\in D} \frac{|\int_\ell A|}{|\ell|^{\gamma}}
\Big|^p
\lesssim
M_p^p\sum_{N \geq 1} \{2^{N(6 -p(\alpha-\gamma)/\omega)}
+
2^{N(12 - p(\beta - \gamma / \omega))}
\}
\;.
\end{equ}
We now take $\omega$ sufficiently large so that $\beta - \gamma/\omega >0$,
and then $p\geq 1$ sufficiently large such that
$6-p(\alpha-\gamma)/\omega < 0$ and $12-p(\beta -\gamma/\omega)<0$.
This ensures that the sum above is finite and the conclusion follows by continuity of $A$.
\end{proof}

\section{Expectation of Wilson loops}
\label{sec:expectation-loop}

We finally take $Z$ as the random model associated to the 3D SYMH equations and a white noise $\xi$ (with the ultraviolet cutoff removed).
In particular, $\Psi$ from \eqref{eq:Psi} solves the SHE $\d_t \Psi = \Delta\Psi + \xi$ with $\Psi_0=0$.

We fix parameters $\alpha,\theta,\gamma,\delta$ satisfying \eqref{eq:init} and $\nu \in (0,1-\delta)$ satisfying \eqref{e:cond-nu}.
Recall $\eta,\mu$ from \eqref{eq:init}
and $\lambda<0$ from \eqref{eq:lambda}.
We further suppose that $\eta>1-2\delta$ and
$\gamma < 2\delta-1$ and $\lambda \in (-1/8,0)$.
Finally, consider $\bar\eta \in [2\theta(\alpha-1),0)$ such that $\eta+\bar\eta>1-2\delta$.

\begin{example}
A possible choice of parameters is
\[
\alpha=\frac12-\eps, \quad \theta=\eps, \quad 
\gamma=\frac12+\eps, \quad \delta=1-\eps,\quad \nu=-\bar\eta = \eps/2
\]
 for $\eps>0$ sufficiently small.
\end{example}

Note that these parameters satisfy the conditions of Lemma \ref{lem:quadnorm_pert} and Corollary~\ref{cor:SHE_quadnorm}
and fall into the setting of Sections \ref{sec:Fs} and \ref{sec:Wilson}.
Moreover, by~\cite[Lem. 2.25]{CCHS_3D},
since $\bar\eta\geq 2\theta(\alpha-1)$,
\begin{equ}[eq:heatgr_CC_bound]
\heatgr{f}_{\alpha,\theta} \lesssim |f|_{\CC^{\bar\eta}}\;.
\end{equ}

Recall from \cite[Lem. 5.18]{CCHS_3D} that 
there exists a space-time distribution $\Psi\d\Psi$ such that \begin{equ}[eq:Psi_d_Psi_def]
\CP \star (\Psi \d \Psi) = \lim_{\eps\downarrow0}\CP \star (\Psi_\eps \d \Psi_\eps)
\;,
\end{equ}
where $\Psi_\eps= \CP\star \bone_+ \xi^\eps$, $\xi^\eps=\moll^\eps*\xi$, for a non-anticipative mollifier $\moll$,
and the limit holds in probability in $\CC^{\kappa}([0,T],\CC^{\kappa})$ for $\kappa>0$ small.

\begin{lemma}\label{lem:SHE_t_eps}
There exists $\eps>0$ such that, for all $p\in [1,\infty)$, uniformly in $t\in (0,1)$,
\begin{equ}
\E \heatgr{\Psi_t}_{\alpha,\theta}^p + \E |\CP_t \star (\Psi \d \Psi)|_{\CC^{\bar\eta}}^p
+ \E \quadnorm{\Psi}^p_{\gamma,\delta} = O(t^{p\eps})\;.
\end{equ}
\end{lemma}

\begin{proof}
By \cite[Prop.~3.7]{CCHS_3D}, there exists $\eps>0$ such that $\E \heatgr{\Psi_t}^2_{\alpha,\theta} = O(t^{2\eps})$.
In particular, $t^{-\eps}\heatgr{\Psi_t}_{\alpha,\theta}$ has a Gaussian tail uniform in $t\in(0,1)$.
Furthermore, by \cite[Lem. 5.18]{CCHS_3D}, there exists $\eps>0$ such that $\E |\CP_t \star (\Psi \d \Psi)|_{\CC^{\bar\eta}}^p=O (t^{p\eps})$.
The claimed bound on $\E \quadnorm{\Psi}^p_{\gamma,\delta}$ is due to Corollary \ref{cor:SHE_quadnorm}.
\end{proof}

Fix henceforth $\kappa>0$ as in Section \ref{sec:SPDE} and $\eps>0$ as in Lemma~\ref{lem:SHE_t_eps}.
For $M>2$ and $t\in(0,1)$, consider the event
\begin{equs}
Q_t &= \{ \|Z\|_{\frac32+2\kappa;O} + |\CP\star \bone_+\xi|_{\CC([-1,3],\CC^{-1/2-\kappa})}
+ |\CP\star (\Psi\d\Psi)|_{\CC([-1,3],\CC^{-2\kappa})}
\\
&\quad
+ t^{-\eps}\heatgr{\Psi_t}_{\alpha,\theta}
+ t^{-\eps} | \CP_t \star (\Psi \d \Psi)|_{\CC^{\bar\eta}} + t^{-\eps}\quadnorm{\Psi_t}_{\gamma,\delta} < M\}
\end{equs}
where $O=[-1,2]\times\T^3$.
Consider initial conditions $X(0),h(0)$ with $|X(0)|_{\CC^3}+|h(0)|_{\CC^3}\lesssim 1$.
Similar to Notation \ref{not:t},
let $\tau = M^{-q}$ for $q\geq 1$ sufficiently large such that SYMH admits a solution on $[0,\tau]$
on the event
\[
\{\|Z\|_{\frac32+2\kappa;O} +  |\CP\star \bone_+\xi|_{\CC([-1,3],\CC^{-1/2-\kappa})}
+ |\CP\star (\Psi\d\Psi)|_{\CC([-1,3],\CC^{-2\kappa})} < M \} \supset Q_t\;,
\]
 which does not depend on $t$.

We henceforth consider $t\in (0,\tau)$.
Let $X=(A,\Phi)$ and $\tilde X = (\tilde A,\tilde\Phi)$ be reconstructions at time $t$ of $\CX$ and $\tilde\CX$ as in Section \ref{sec:Wilson} for generic smooth $X(0)$ and $h(0)$ with $|X(0)|_{\CC^3}+|h(0)|_{\CC^3}\lesssim 1$.
Recall the line $\ell \colon [0,1]\to\T^3$, $\ell(x)=(x,0,0)$.
The following is the main result of this section.

\begin{lemma}\label{lem:E_W_ell}
For all $r>0$ sufficiently small, there exists $\beta>0$ such that, for $t\ll 1$, $s=t^{\beta}$, and $|X(0)|_{L^\infty} \lesssim t^{r}$,
  \begin{equs}\label{eq:E_W_diff}
{}&\E W_\ell (\CF_s \tilde A) \bone_{Q_t}
- \E W_\ell(\CF_s A)\bone_{Q_t}
=
\P[Q_t] \Big\{ t\Trace\Big( \int_\ell c h(0) \Big)
\\
&\quad + t \Trace \Big( \int_{[0,1]^2} \mrd \ell_{A(0)}(x_1) \mrd \ell_{c h(0)}(x_2) \Big)
+ O(t^{1+r+ \beta/6} + t^{1+3r/2} + t^{1+r + \nu \beta}) \Big\}\;,
\end{equs}
where the proportionality constants are $O(M^k)$ for some $k\geq 0$.
\end{lemma}

\begin{proof}
Assume we are on the event $Q_t$.
Recall $u=s^\lambda \heatgr{A}_{\alpha,\theta}$, $\lambda < 0$, and $\eps>0$ from Lemmas~\ref{lem:Wilson_expansion} and~\ref{lem:SHE_t_eps}.
Let us take $r < \frac{\eps}{2} \wedge\frac{1}{10}$.
It follows from \eqref{eq:heatgr_CC_bound}, the expansion of $A$ from \eqref{eq:a_expansion}, and $|A(0)|_{L^\infty}\lesssim t^r$,
that
\begin{equ}[eq:heatgr_a]
\heatgr{A}_{\alpha,\theta} \leq |A(0)|_\infty + \heatgr{\Psi_t}_{\alpha,\theta} + \heatgr{\CP_t \star (\Psi \d \Psi)}_{\alpha,\theta} + O(t^{1/4-3\kappa/2}) \lesssim t^{r}\;.
\end{equ}
We now take
\begin{equ}
\beta=-\frac{r}{4\lambda}>0\;,
\end{equ}
so that $s^\lambda = t^{\lambda \beta} = t^{-r/4}$
and thus
\begin{equ}[e:u-to-t]
u = s^{\lambda} \heatgr{A}_{\alpha,\theta} \lesssim s^\lambda t^{r}
= t^{3r/4}\;.
\end{equ}
Since $\eta> -2/3$, we have $s^{\eta/2+1/2} \leq s^{1/6} = t^{\beta/6}$.

Recall that $|f|_{\CC^\eta} \lesssim \heatgr{f}_{\alpha,\theta}$ due to \eqref{eq:eta_heatgr}.
Therefore, by \eqref{eq:a_expansion} and Lemma \ref{lem:quadnorm_pert},
\begin{equ}[eq:quadnorm_a]
\quadnorm{A}_{\gamma,\delta} \lesssim \quadnorm{\Psi_t}_{\gamma,\delta} + |\Psi|_{\CC^\eta} \{|A(0)|_{L^\infty} + |\CP_t (\Psi\d \Psi)|_{\CC^{\bar\eta}} + O(t^{1/4-3\kappa/2})\}
\lesssim t^r\;.
\end{equ}
Hence, combining \eqref{eq:heatgr_a} and \eqref{eq:quadnorm_a}, we have $\SNorm(A)\lesssim t^{r}$.
Note that our choice $s = t^\beta$ in particular satisfies the condition $s\ll 1\wedge \Poly(\SNorm(A)^{-1})$ of Lemma \ref{lem:Wilson_expansion} for all $t>0$ sufficiently small.

The conclusion now follows from Lemma \ref{lem:Wilson_expansion} by taking $L$ sufficiently large so that all the errors in \eqref{eq:reg_W_ell} are bounded 
as
\begin{equs}
t^{5/4-7\kappa/2} + ts & \lesssim t^{1+3r/2} \;,
\\
 ts^\lambda \big( \heatgr{\Psi_t}_{\alpha,\theta} + \heatgr{\CP_t \star (\Psi \d \Psi)}_{\alpha,\theta} \big) 
& \lesssim 
 s^\lambda t^{1+\eps}
 \lesssim t^{1+3r/2} \;,
 \\
 |A(0)|_{L^\infty} s^{\eta/2+\frac12} t
&\lesssim t^{r+1} t^{(\frac\eta2+\frac12)\beta}
\lesssim 
t^{1+r+\beta/6} \;, 
\\
s^\nu t(\SNorm(A)+\SNorm(A)^3)
&\lesssim t^{1+r+\nu \beta}\;,
\\
 t(u^2 + u^{L-1})
&\lesssim
t^{1+3r/2} \;,
\end{equs}
where we used $\eta>-2/3$, \eqref{e:u-to-t},
and $\lambda \in (-1/8,0)$ thus $s=t^\beta < t^{2r}$.
The terms $s^{\nu L} + u^L + u^{L+1} + t^2 u$ in \eqref{eq:reg_W_ell} 
are clearly even smaller.
\end{proof}

\section{Finishing the proof}
\label{sec:initial}

In this section, we conclude the proof of Proposition \ref{prop:A_tilde_A}.
Given the ingredients in the previous sections, the proof is close in spirit to that of
\cite[Prop.~8.5]{chevyrev2023invariant},
but has a few important differences.
In general, our goal is to argue that the right-hand side of \eqref{eq:E_W_diff} in Lemma~\ref{lem:E_W_ell}
indeed gives us a sufficient lower bound. 
To achieve such a lower bound, we need to use the
two explicit terms on the right-hand side of \eqref{eq:E_W_diff};
these are the  `good terms' advertised in Section~\ref{sec:idea}.
In particular,
we will show that for $t$ small, if $\Trace \big( \int_\ell c h(0) \big)\neq 0$, then
 we obtain a lower bound of order $t$;
and if $\Trace \big( \int_\ell c h(0) \big)=0$, in which case the strategy will have some key difference with \cite{chevyrev2023invariant}, then we obtain a lower bound of order $t^{1+r}$ from the second term on the right-hand side of \eqref{eq:E_W_diff}.
The difference arises from the following reason.
In \cite[Prop.~8.5]{chevyrev2023invariant}, since the space dimension is two, 
the `big-O' remainder term therein is $O(t^{2+})$ which is much smaller than  $O(t^{1+r+})$ here.
This allows one to simply choose $A(0)=0$ therein and then a term {\it quadratic} in $h(0)$ is of order $t^2$ which  dominates the $O(t^{2+})$ remainder.
In the argument here we have to be more delicate and make use of the {\it cross-term} between $A(0)$ and $h(0)$,
and in fact we will take non-zero $A(0)$, specifically
$A(0) = t^r c h(0)$, with which we obtain a term with an explicit lower bound of order $t^{1+r}$.

Suppose we are in the setting of Proposition \ref{prop:A_tilde_A}. We also follow the setting and notation of Section \ref{sec:expectation-loop}.
We recall (a special case of) \cite[Lem.~B.2]{chevyrev2023invariant},
which follows from the Chow--Rashevskii theorem.

\begin{lemma}\label{lem:curve_selection}
Consider non-zero $j \in L(\mfg,\R)$.
Then there exists $\zeta \in \CC^\infty([0,1], \mfg)$
such that
\begin{enumerate}[label=(\roman*)]
\item $\dot\zeta=0$ on $[0,\frac14]$ and $[\frac34,1]$,
\item
$\zeta(0)=0$ and $j\zeta(1)\neq 0$,
\item\label{pt:targets}
$L^\zeta(1)=\id$ where $L^\zeta\in\CC^\infty([0,1],G)$ solves
$
\mrd L^\zeta = (\mrd \zeta) L^\zeta$, $L^\zeta(0)=\id
$.
\end{enumerate}
\end{lemma}

\begin{proof}[of Proposition \ref{prop:A_tilde_A}]
It suffices to consider $r>0$ small as in Lemma \ref{lem:E_W_ell}.
Consider for now generic $g(0)\in \mfG^\infty$ and $\tilde x \in \CC^\infty(\T^3,E)$ with $|g(0)|_{\CC^3}+|\tilde x|_{\CC^3}\lesssim 1$.
Consider $M\gg 1$ and $t = M^{-q'}$, where $q'>q$ so that $t<\tau$ for $\tau = M^{-q}$ as in Section \ref{sec:expectation-loop}.
Then, by Lemma \ref{lem:SHE_t_eps} and Markov's inequality, $\P[Q_t^c]\lesssim M^{-2q'} = t^2$.
In particular,
$\E W_\ell (\CF_s A) = \E W_\ell (\CF_s A) \bone_{Q_t} + O(t^2)$,
and similarly for 
$\E W_\ell(\CF_s \tilde A)\bone_{Q_t}$, so it suffices to show that, for $t>0$ sufficiently small,
\begin{equ}[eq:E_W_diff_suffice]
|\E \{W_\ell (\CF_s \tilde A) - W_\ell(\CF_s A)\}\bone_{Q_t}| \gtrsim t^{1+r}\;.
\end{equ}
Let $k$ be as in Lemma \ref{lem:E_W_ell} so that the term $O(t^{1+r+ \beta/6} + t^{1+3r/2} + t^{1+r + \nu \beta})$ appearing in \eqref{eq:E_W_diff} 
is bounded from above by $M^k(t^{1+r+ \beta/6} + t^{1+3r/2} + t^{1+r + \nu \beta})$.
Then, by taking $q'$ large, the final quantity is of order $o(t^{1+r})$ with a proportionality constant that does not depend on $M$.

Recall $c\in L(\mfg^3)$ with $c_1\neq 0 \in L(\mfg^3,\mfg)$ in the statement of Proposition \ref{prop:A_tilde_A}.
We write $c_1(A_1,A_2,A_3) = \sum_{i=1}^3 c_1^{(i)}A_i$ where $c_1^{(i)}\in L(\mfg)$.
It suffices to consider the cases $c_1^{(1)}\neq 0$ and $c_1^{(2)}\neq 0$.

\textbf{Case 1: $c_1^{(1)}\neq 0$.}
Let $\zeta\in \CC^\infty([0,1],\mfg)$ be as in Lemma~\ref{lem:curve_selection} with $j=\jmath\circ c_1^{(1)}$ for any $\jmath\in L(\mfg,\R)$ that makes $j$ non-zero.
In particular $c^{(1)}_1\zeta(1)\neq 0$.

Define $u\in\CC^\infty(\T^3,G)$ by $u(x,y,z)=L^\zeta(x)$ for $(x,y,z)\in [0,1)^3$, where $L^\zeta$ is as in Lemma~\ref{lem:curve_selection}\ref{pt:targets}
and we recall that we identify $\T^3$ with $[0,1)^3$.
Remark that $\d_2 u = \d_3 u =0$ and that indeed $u\colon \T^3\to G$ is smooth.
Define $g(0) = u$, $h(0) = (\mrd u)u^{-1}$, and $\tilde x = (A(0),0)\in\CC^\infty(\T^3,\mfg^3\oplus\higgsvec)$ where
\begin{equ}
A(0) = t^{r}ch(0)\;.
\end{equ}
In particular, $h_i \equiv 0$ for $i=2,3$ and, by our choice of $\ell$,
\[
\int_\ell c h(0)  =\int_\ell c_1^{(1)} h_1(0) = c_1^{(1)} \zeta(1)\;.
\]

Observe now that, if $\Trace(c^{(1)}_1\zeta(1))\neq 0$,
then Lemma \ref{lem:E_W_ell} implies
\begin{equ}
|\E W_\ell (\CF_s \tilde A) \bone_{Q_t}
- \E W_\ell(\CF_s A)\bone_{Q_t}| \gtrsim t\;,
\end{equ}
which in turn clearly implies \eqref{eq:E_W_diff_suffice} and we are done.

If, on the other hand, $\Trace(c^{(1)}_1\zeta(1)) = 0$, then the first term on the right-hand side of \eqref{eq:E_W_diff} vanishes.
Therefore, recalling that $\int_{[0,1]}\mrd \ell_{ch(0)}(x) = c_1^{(1)}\zeta(1)\neq0$ and $A(0)=t^r ch(0)$, we are left with
\begin{equ}
|\E W_\ell (\CF_s \tilde A) \bone_{Q_t}
- \E W_\ell(\CF_s A)\bone_{Q_t}|
\gtrsim t^{1+r} \big|\Trace \big( \{c_1^{(1)}\zeta(1)\}^2 \big)\big| - o(t^{1+r}) \gtrsim t^{1+r}
\end{equ}
as required.

\textbf{Case 2: $c^{(2)}_1\neq 0$.} This case is easier and does not require Lemma \ref{lem:curve_selection}.
It is similar to Case 2 of the proof of \cite[Prop.~8.5]{chevyrev2023invariant}, but we give the details for completeness.
Define $u\in \CC^\infty(\T^3,G)$ as follows.
Let $X\in\mfg$ such that $c^{(2)}_1 X \neq 0$.
Consider smooth $\psi\colon [-\frac12,\frac12]\to [0,1]$
such that $\psi(0)=\psi(y)=0$ for all $|y|\geq \frac14$
and $\dot \psi(0)=1$.
We then define $u\colon \T^3 \to G$ by
\begin{equ}
u(x,y,z) =
\begin{cases}
e^{\psi(y)X} &\quad\text{ if } y\in[0,\frac14]\;,\\
1 &\quad\text{ if } y\in[\frac14,\frac34]\;,\\
e^{\psi(y-1)X} &\quad\text{ if } y\in[\frac34,1]\;.
\end{cases}
\end{equ}
Then $u$ is smooth and $\d_1 u = \d_3 u =0$.
Furthermore $h=(\mrd u)u^{-1}$ satisfies $h_1=h_3\equiv 0$ and
\[
h_2(x,0,0)\eqdef (\d_2 u) u^{-1}(x,0,0) = X\qquad  \mbox{for all  } x\in [0,1]
\]
where we used  $\dot \psi(0)=1$.
In particular, recalling the definition of $\ell$, one has 
\[
\int_\ell ch = \int_0^1 c^{(2)}_1 h_2(x,0,0) \mrd x = c^{(2)}_1 X\;.
\]
The conclusion follows exactly as in Case 1 upon subdividing into the cases $\Trace c^{(2)}_1X\neq 0$ and $\Trace c^{(2)}_1X = 0$.
\end{proof}

\section{Symbolic index}

We collect in this appendix commonly used symbols of the article, together
with their meaning and, if relevant, the page where they first occur.

\begin{center}
\renewcommand{\arraystretch}{1.1}
\begin{longtable}{lll}
\toprule
Symbol & Meaning & Page\\
\midrule
\endfirsthead
\toprule
Symbol & Meaning & Page\\
\midrule
\endhead
\bottomrule
\endfoot
\bottomrule
\endlastfoot
$C^\eps_{\YM / \Higgs}$  & BPHZ `constants' $C^\eps_{\YM} \in L_G(\mfg)$ and $C^\eps_{\Higgs}\in L_G(\higgsvec)$ &  \pageref{page:BPHZ} \\
$\mathring C_{\A / \Phi}$  & Arbitrary operators $\mathring C_{\A} \in L(\mfg^3)$ and $\mathring C_\Phi \in L(\higgsvec)$ &  \pageref{page:C_bare} \\
$C^\eps_{\A / \Phi}$  & Total renormalisation operators $C^\eps_{\A / \Phi} = C^\eps_{\YM / \Higgs} + \mathring C_{\A / \Phi}$ &  \pageref{page:C_eps} \\
$E$  & YM-Higgs target space $E= \mfg^3\oplus \higgsvec$ &  \pageref{page:E} \\
$\CF_s$  & YM heat flow $\{\CF_s(a)\}_{s\geq 0}$  &  \pageref{page:CF} \\
$F^\sol$  & Space of paths $f\colon [0,1] \to F\sqcup \{\skull\}$ with possible blow-up &  \pageref{page:sol} \\
$G$ & Compact Lie group $G \subset \U(N)$ &  \pageref{page:G} \\
$\mfG^\rho$ & Gauge group $G^\rho = \CC^\rho(\T^3,G)$ &  \pageref{page:mfG} \\
$\mfg$ & Lie algebra $\mfg\subset\mfu(N)$ of $G$ &  \pageref{page:mfg} \\
$g\act$ & Action of gauge group &  \pageref{e:gauge-transformation} \\
$L(F)$ & Linear operators $F\to F$ &  \pageref{page:L} \\
$L_G$ & Linear operators commuting with action of $G$ &  \pageref{page:L} \\
$\CN_s$ & Quadratic functional $\CN_s(A) \eqdef P_s A\otimes \nabla P_s A$ & \pageref{eq:CN_def} \\
$O$  & Space-times set $O=[-1,2]\times\T^3$  &  \pageref{page:O} \\
$\CP$  & Integration against heat kernel on modelled distributions  &  \pageref{page:CP} \\
$\CP\star$  & Integration against heat kernel of space-time distributions  &  \pageref{eq:P_star_def} \\
$P_t$  & Heat semigroup $(P_t)_{t\geq 0}$  &  \pageref{page:P}\\
$\SYMH_t$  & Solution $\SYMH(C,x)\in\state^\sol$ of the SYMH at time $t \in [0,1]$ &  \pageref{page:SYMH}\\
$\state$  & State space from \cite{CCHS_3D} &  \pageref{def:state}\\
$\SState$  & Refined state space with better bounds on $\CN_s$ &  \pageref{def:SState}\\
$\Sigma$  & Metric / size functional on $\SState$ &  \pageref{def:SState}\\
$\higgsvec$  & Target space of Higgs field, a real Hilbert space  &  \pageref{page:higgsvec}\\
$W_\ell$  & Wilson loop observable &  \pageref{page:W_loop}\\
$\bullet^\ym$ & $\mfg^3$ (i.e. YM) component of $\mfg^3\oplus\higgsvec$-valued field&  \pageref{page:ym}\\
$\Omega_{\gr\gamma}$ & Space of distributions with finite `growth' norm $|\cdot|_{\gr\gamma}$ &  \pageref{page:gr}\\
$\quadnorm{\bullet}_{\gamma,\delta}$ & Time-weighted norm on $\CN_s(\bullet)$ &  \pageref{page:quadnorm}\\
\end{longtable}
\end{center}

{\bf Data Availability Statement.} Data sharing not applicable to this article as no datasets were generated or analysed during the current study.

\endappendix

\bibliographystyle{./Martin}
\bibliography{./refs}

\end{document}